\documentclass[final]{siamltex}
\usepackage{graphicx}
\usepackage{leftidx}
\usepackage{mathrsfs}
\usepackage{amssymb}
\usepackage{amsmath}
\usepackage{cite}
\usepackage{float}
\usepackage{multirow}
\usepackage{color}

\def\af{\alpha}

 \usepackage{leftidx}

\title{SUPERCONVERGENCE POINTS FOR THE SPECTRAL INTERPOLATION OF RIESZ FRACTIONAL DERIVATIVES\thanks{This work is supported in part by the National Natural Science Foundation of China under grants NSFC 11471031, NSFC 91430216, and NSAF U1530401; the US National Science Foundation through grant DMS-1419040, the Natural Science Youth Foundation of Jiangsu Province of China (No.SBK20160660); and the Fundamental Research Funds for the Central Universities of China (No.2242016K41029).} }  

\author{Beichuan Deng\thanks{Department of Mathematics, Wayne State University, Detroit, MI 48202, USA. (\tt beichuan.deng@wayne.edu).}
\and Zhimin Zhang\thanks{Corresponding author. Beijing Computational Science Research Center, Beijing 100193, China ({\tt zmzhang@csrc.ac.cn}); and Department of Mathematics, Wayne State University,  Detroit, MI 48202, USA, (\tt zzhang@math.wayne.edu).}
\and Xuan Zhao\thanks{School of Mathematics, Southeast University, Nanjing 210096, China.
(\tt xuanzhao11@seu.edu.cn).}}

\begin{document}
\maketitle

\begin{abstract}
In this paper, superconvergence points are located for the approximation of the Riesz derivative of order $\alpha$ using classical Lobatto-type polynomials when $\alpha \in (0,1)$ and generalized Jacobi functions (GJF) for arbitrary $\alpha > 0$, respectively. For the former, superconvergence points are zeros of the Riesz fractional derivative of the leading term in the truncated Legendre-Lobatto expansion. It is observed that the convergence rate for different $\af$ at the superconvergence points is at least $O(N^{-2})$ better than the optimal global convergence rate. Furthermore, the interpolation is generalized to the Riesz derivative of  order $\alpha > 1$ with the help of GJF, which deal well with the singularities. The well-posedness, convergence and superconvergence properties are theoretically analyzed. The gain of the convergence rate at the superconvergence points is analyzed to be $O(N^{-(\alpha+3)/2})$ for $\alpha \in (0,1)$ and $O(N^{-2})$ for $\alpha > 1$. Finally, we apply our findings in solving model FDEs and observe that the convergence rates are indeed  much better at the predicted superconvergence points.
\end{abstract}



\pagestyle{myheadings}
\thispagestyle{plain}
\markboth{ }{Superconvergent points for Riesz Fractional Derivatives}
\section{Introduction}





.Over the last two decades,  the theory of fractional differential equations (FDEs) has been extensively studied. Especially, the Riesz fractional derivative, which appears frequently in spatial fractional models (such as various diffusion models), has been widely studied. Some finite-difference based numerical schemes for approximating Riesz fractional derivatives and solving linear or nonlinear Riesz FDEs are presented in \cite{Deng+2012+AM,Shen+2014+IMA, XuanSun+2014+SISC}. There are also works that apply finite element methods to solve Riesz FDEs; some of them concentrate on theoretical analysis, \cite{Bu+2014+JCP,Roop+2006+AM}, and some are mainly about improving algorithms, such as fast algorithms \cite{Lei+2013+JCP,HPang+2012+JCP}.



{  On the other hand, spectral methods are promising candidates for solving FDEs since their global nature fits well with the nonlocal definition of fractional operators. Using integer-order orthogonal polynomials as basis functions, spectral methods \cite{Li+2012+FCAA,Li+2009+SINUM,Xu+2014+JCP,Zeng+2014+SINUM,Zheng+2015+SISC} help enormously with the alleviation of the memory cost for the discretization of fractional derivatives. The authors of \cite{Chen+2014+ar,Huang+2014+ar,Zayernouri+2013+JCP,Zayernouri+2014+SISC}
 designed suitable bases to deal with singularities, which usually appear in fractional problems. In particular, Mao, Chen, and Shen \cite{Mao+2015+ANM} proposed a spectral Petrov-Galerkin method, which is based on generalized Jacobi functions, for solving Riesz FDEs, and provided rigorous error analysis.

In this work, we study the superconvergence phenomenon for some spectral interpolation of the Riesz fractional derivative.
In the literature, superconvergence of the $h$-version finite element method has been well studied and understood, see, e.g.,  \cite{Lin+2006+Bj,LBW+1995+note}, while there have been some relatively recent  superconvergence studies of polynomial spectral interpolation and spectral methods in the case of integer-order derivatives, see, e.g., \cite{Zhang+2008+JSC,Zhang+2012+SINUM,Xie+2012+MACOM,LWang+2014+JSC}.
As for fractional-order derivatives, Zhao and Zhang studied the Riemann-Liouville case recently \cite{Zhao+2016+SISC} and found systematically some superconvergence points in the spirit of an earlier work on integral-order spectral methods \cite{Zhang+2012+SINUM}.


A major difficulty in the investigation of superconvergence of spectral methods for fractional problems, compared with integer-order derivatives, is the nonlocality of fractional operators and the complicated forms of fractional derivatives.
The second challenge is the construction of a good basis for a spectral scheme. Given a suitable  basis, one can then begin the analysis of the approximation error in order to locate the superconvergence points.

One objective of this work is to consider Lobatto-type polynomial interpolants of a sufficiently smooth function, and identify those points where the values of the Riesz fractional derivative of order $\alpha$ are superconvergent.
Note that $x=\pm1$ are interpolation points for the Legendre-Lobatto interpolation; this fact guarantees that after taking the derivative of order $\alpha\in (0,1)$, the global error doesn't blow up.
In this case, superconvergence points are zeros of the Riesz fractional derivative of the corresponding Legendre-Lobatto polynomials. Furthermore, according to a series of numerical experiments, we observe that the convergence rate is at least $O(N^{-2})$ better than the global convergence rate. 

Comparing with the Riemann-Liouville fractional derivative, the main difficulty in studying the Riesz case is to deal with the left and right fractional derivatives simultaneously, see the definition of (2.2)--(2.3).
Fractional derivatives of order $\alpha>1$ create stronger singularities at $x=\pm1$, which restrain the use of Lobatto-type polynomials. To handle the stronger singularity, the  fractional interpolation using generalized Jacobi functions (GJF) is introduced (in Section 4), and its well-posedness, convergence and superconvergence properties are theoretically analyzed. When $0<\alpha<2$, if the given function is GJF-interpolated at zeros of the Jacobi polynomial $P_{N+1}^{\frac{\alpha}{2},\frac{\alpha}{2}}(x)$, then the superconvergence points for a fractional derivative of order $\alpha$ are exactly the interpolation points. Moreover, the convergence rate at the superconvergence points is $O(N^{-\frac{\alpha+3}{2}})$ and $O(N^{-2})$ higher than the global convergence rate for $0<\alpha<1$ and $1<\alpha<2$, respectively.

To demonstrate the usefulness of our discovery of these superconvergence points, we use the GJF as a basis to solve a model fractional differential equation by both Petrov-Galerkin and spectral collocation methods. We observe that convergence rates at the predicted superconvegrence points are indeed much better than the best possible global rates.

The organization of this paper is as follows. In Section 2, definitions and properties of fractional derivatives, Jacobi polynomials, generalized Jacobi functions, and Gegenbauer polynomials are introduced. Section 3 is about the Legendre-Lobatto polynomial interpolation and Section 4 deals with the GJF fractional interpolation, along with numerical examples. Section 5 considers some applications of superconvergence theory. Finally, we draw some conclusions in Section 6.  \\

\section{Preliminaries}
We begin with some basic definitions and properties. Throughout the paper, $\mathbb{Z}^+$ denotes the set of all positive integers, $\mathbb{N}$ denotes the set of all nonnegative integers. $\mathcal{M}_{n}(\mathbb{R})$ denotes the space of all of $n \times n$ matrices defined on real number field.

\subsection{Definitions and Properties of Fractional Derivatives}
First, we recall the definitions and properties of Riesz fractional derivatives.

\begin{definition} Let $\gamma \in (0,1)$, the left and right fractional integral are defined respectively, as follows:
$${_{-1}I_x^{\gamma}}u(x):=\frac{1}{\Gamma(\gamma)} \int_{-1}^x \frac{u(\tau)}{(x-\tau)^{1-\gamma}} d\tau, \ x \in (-1,1], $$
$${_{x}I_1^{\gamma}}u(x):=\frac{1}{\Gamma(\gamma)} \int_{x}^1 \frac{u(\tau)}{(\tau-x)^{1-\gamma}} d\tau, \ \ x \in [-1,1).$$
Then for $\alpha \in (k-1,k)$, where $k \in \mathbb{Z}^{+}$, the left and right Riemann-Liouville derivatives are defined respectively by:
$$_{-1}D_x^{\alpha} u(x) = D^k (_{-1}I_x^{k-\alpha}u(x)),$$
$$_{x}D_1^{\alpha} u(x) = (-1)^k D^k (_xI_1^{k-\alpha}u(x)),$$
where $D^k:=\frac{d^k}{dt^k}$ is the $k$-th (weak) derivative.
\end{definition}

\begin{definition}
Let $\gamma \in (0,1)$, the one dimensional Riesz potentials are defined as follows:
$$I^{\gamma}_o u(x):= \frac{c_1}{\Gamma(\gamma)}\int _{-1}^1 \frac{sign(x-\tau) u(\tau)}{|x-\tau|^{1-\gamma}} d\tau = c_1 (_{-1}I_x^{\gamma}-{_x I_1^{\gamma}})u(x)$$
$$I^{\gamma}_e u(x):= \frac{c_2}{\Gamma(\gamma)}\int _{-1}^1 \frac{u(\tau)}{|x-\tau|^{1-\gamma}} d\tau = c_2 (_{-1}I_x^{\gamma}+{_x I_1^{\gamma}})u(x)$$
where $sign(x)$ is the sign function, $c_1=\frac{1}{2 \sin(\pi\gamma/2)}$, $c_2=\frac{1}{2 \cos(\pi\gamma/2)}$. Then for $\alpha \in (k-1,k)$, we can therefore define the Riesz fractional derivative:
\begin{eqnarray}
^R D^{\alpha}u(x):=
\left\{\begin{array}{ll}
D^k I_o^{k-\alpha}u(x)=c_1 (_{-1}D_x^{\alpha}+ {_{x}D_1^{\alpha}} )u(x), \ k \ is \ odd\\
D^k I_e^{k-\alpha}u(x)=c_2 (_{-1}D_x^{\alpha}+ {_{x}D_1^{\alpha}} )u(x), \ k \ is \ even
\end{array}\right.
\label{def22}
\end{eqnarray}
\end{definition}

\begin{definition}
For any positive real number $\alpha \in (k-1,k)$, we define:
\begin{eqnarray}
^R D^{\alpha}_o u(x):=D^k I_o^{k-\alpha}u(x)=c_1 (_{-1}D_x^{\alpha}+ (-1)^k {_{x}D_1^{\alpha}} )u(x),
\label{def231}
\end{eqnarray}
and
\begin{eqnarray}
^R D^{\alpha}_e u(x):=D^k I_e^{k-\alpha}u(x)=c_2 (_{-1}D_x^{\alpha}+ (-1)^k {_{x}D_1^{\alpha}} )u(x).
\label{def232}
\end{eqnarray}
\end{definition}
Let us recall the Leibniz rule for fractional derivatives.

\begin{lemma}(see \cite{Podlubny+1999}, Chap.2)
Let $\alpha \in \mathbb{R}^+$, $n \in \mathbb{Z}^+$, and $\alpha \in (n-1,n)$. If both $f(x)$ and $g(x)$ along with all their derivatives are continuous in $[-1,1]$, then the Leibniz rule for the left Riemann-Liouville fractional differentiation takes the following form
\begin{eqnarray}
_{-1}D_x^{\alpha}[f(x)g(x)]=\sum_{k=0}^{\infty} \frac{\Gamma(\alpha+1)}{\Gamma(k+1)\Gamma(\alpha-k+1)} (_{-1}D_{x}^{\alpha-k} f(x)) g^{(k)}(x).
\label{lemma24}
\end{eqnarray}
\end{lemma}

By changing variables, we can derive the Leibniz rule for the right Riemann-Liouville derivative.
\begin{lemma}(see\cite{FIAD+Samko}, Chap.15)
Under the same conditions as Lemma 2.4, the Leibniz rule for the right Riemann-Liouville differentiation takes the following form
\begin{eqnarray}
_xD_1^{\alpha}[f(x)g(x)]=\sum_{k=0}^{\infty} (-1)^k \frac{\Gamma(\alpha+1)}{\Gamma(k+1)\Gamma(\alpha-k+1)} (_xD_1^{\alpha-k} f(x)) g^{(k)}(x).
\label{lemma25}
\end{eqnarray}
\end{lemma}

\subsection{Jacobi Polynomials and Generalized Jacobi Functions}
.We start from the definition of Generalized Jacobi Functions and Gegenbauer polynomials.
\\
\begin{definition}
Let $\alpha >-1$, the Generalized Jacobi Functions (GJF) is defined as follows:
$$\mathcal{J}_n^{-\alpha,-\alpha}(x):=(1-x^2)^{\alpha}P_n^{\alpha,\alpha}(x)$$
where $P_n^{\alpha,\alpha}(x)$ is the $n$-th Jacobi polynomial with respect to the weight function $\omega(x)=(1-x)^{\alpha}(1+x)^{\alpha}$.
\end{definition} \\

\begin{definition}
We define the Gegenbauer polynomials by Jacobi polynomials:
$$C_n^{(\lambda)}(x):=c_{(\lambda,n)}P_n^{(\lambda-\frac{1}{2},\lambda-\frac{1}{2})}(x)=\frac{(2 \lambda)_n}{(\lambda+\frac{1}{2})_n}P_n^{(\lambda-\frac{1}{2},\lambda-\frac{1}{2})}(x).$$
\end{definition} \\
The Gegenbauer polynomials have the following properties:
\begin{eqnarray}
\frac{d}{dx}C_n^{(\lambda)}(x)=2 \lambda C_{n-1}^{(\lambda+1)}(x),
\label{gegenprop1}
\end{eqnarray}
and
\begin{eqnarray}
\max_{-1\leqslant x \leqslant 1} \{ C_n^{(\lambda)}(x) \} \left\{\begin{array}{ll}
=C_n^{(\lambda)}(1)=\frac{\Gamma(n+2\lambda)}{\Gamma(n+1) \Gamma(2 \lambda)} \sim \frac{N^{2\lambda-1}}{\Gamma(2\lambda)}, \ \hbox{when} \ \lambda>0 \\
\leqslant D_{\lambda} N^{\lambda-1}, \ \ \ \ \ \ \ \ \ \hbox{when} \ -\frac{1}{2} <\lambda <0
\end{array}\right.
\label{gegenprop2}
\end{eqnarray}
where $D_{\lambda}$ is a positive constant independent of $n$; and
\begin{eqnarray}
\vert C_{n+1}^{(\lambda)}(z) \vert \geqslant  \frac{n^{\lambda-1} \rho^{n+1}}{2 \Gamma(\lambda)} (1+\rho^{-2})^{-\lambda}, \ \forall z \in \mathcal{E}_{\rho},
\label{gegenprop3}
\end{eqnarray}
where $\mathcal{E}_{\rho}$ is the $Berstein \ ellipse$ defined in (\ref{bersteindef}). When we set $\lambda = \frac{\alpha+1}{2}$, we have:
$$P_{N+1}^{\frac{\alpha}{2},\frac{\alpha}{2}}(1) =\frac{\Gamma(N+2+\frac{\alpha}{2})}{\Gamma(N+2)\Gamma(\frac{\alpha}{2}+1)} \sim N^{\frac{\alpha}{2}},$$
and
$$C_{N+1}^{(\frac{\alpha+1}{2})}(1) = \frac{\Gamma(N+2+\alpha)}{\Gamma(N+2) \Gamma(\alpha+1)}\sim N^{\alpha},$$
consequently
\begin{eqnarray}
c_{(\frac{\alpha+1}{2},N+1)} \sim N^{\frac{\alpha}{2}},
\label{property29}
\end{eqnarray}
where $C_{N+1}^{(\frac{\alpha+1}{2})}(x)=c_{(\frac{\alpha+1}{2},N+1)} P_{N+1}^{\frac{\alpha}{2},\frac{\alpha}{2}}(x)$. \\

Then, the following lemma shows the connection between the GJF and Riesz fraction derivatives.

\begin{lemma}(see \cite{Mao+2015+ANM}, Theorem 2)
Let $\alpha \in (k-1,k)$, $k \in \mathbb{Z}^+$, then we have
\begin{eqnarray}
I_{\nu}^{k-\alpha} \mathcal{J}_n^{-\frac{\alpha}{2},-\frac{\alpha}{2}}(x)=C(k)\frac{\Gamma(n+\alpha+1-k)}{2^{-k}n!}P_{n+k}^{\frac{\alpha}{2}-k,\frac{\alpha}{2}-k}(x);
\end{eqnarray}
moreover, for $m=0,1,\ldots,k-1$,
\begin{eqnarray}
^R D_{\nu}^{\alpha-m} \mathcal{J}_n^{-\frac{\alpha}{2},-\frac{\alpha}{2}}(x)=C(k)\frac{\Gamma(n+\alpha+1-m)}{2^{-m}n!}P_{n+m}^{\frac{\alpha}{2}-m,\frac{\alpha}{2}-m}(x);
\label{lemma211}
\end{eqnarray}
where $\nu=o$, $C(k)=(-1)^{\frac{k-1}{2}}$, if $k$ is odd; $\nu=e$, $C(k)=(-1)^{\frac{k}{2}}$, if $k$ is even. Especially:
\begin{eqnarray}
^R D^{\alpha} \mathcal{J}_n^{-\frac{\alpha}{2},-\frac{\alpha}{2}}(x)=C(k)\frac{\Gamma(n+\alpha+1)}{n!}P_{n}^{\frac{\alpha}{2},\frac{\alpha}{2}}(x).
\label{lemma212}
\end{eqnarray}
\label{lemma28}
\end{lemma}

\section{Legendre-Lobatto Interpolation for $0<\alpha<1$}
According to Definition 2.2, the Riesz fractional derivative of order $\alpha$ is equivalent to the two-sided Riemann-Liouville fractional derivatives of the same order. This leads to singularities at $x=\pm 1$. In order to get rid of the singularities, we consider interpolating $u(x)$ by Lobatto-type polynomials, in particular the Legendre-Lobatto polynomials. By doing so, $x= \pm1$ are zero points of multiplicity 1 of the error $u(x)-u_N(x)$ since $x=\pm 1$ are two of the interpolation points. This guarantees that, after taking derivatives of order $\alpha \in (0,1)$, the global error is finite. In this section, we always assume that $u(x)$ is analytic on $[-1,1]$, and can be analytically extended to a certain $Berstein \ ellipse$.

\subsection{Interpolation of Analytic Functions}
Let $\{ x_i \}_{i=0}^N$ be the $(N+1)$ interpolation points, where $-1 \leqslant x_0<x_1< \cdots <x_N \leqslant 1$. Then we define
\begin{eqnarray}
\omega_{N+1}(x)=\prod_{i=0}^N (x-x_i).
\label{def31}
\end{eqnarray}
If $\{ x_i \}_{i=0}^N$ is the set of zero points of $(N+1)$ degree Legendre-Lobatto polynomial, then
\begin{eqnarray}
\omega_{N+1}(x) \eqsim L_{N-1}(x)-L_{N+1}(x),
\label{def32}
\end{eqnarray}
where $L_{n}(x)$ represents the Legendre polynomial of degree $n$, and the right hand side is exactly the Legendre-Lobatto polynomial of degree $(N+1)$.

Suppose that $u(x)$ is analytic on $[-1,1]$, it is well known that $u(x)$ can be analytically extended to a domain enclosed by the so-callded $Berstein$ $ellipse$, with the foci $\pm 1$:
\begin{eqnarray}
\mathcal{E}_{\rho}:=\{z: z=\frac{1}{2}(\rho e^{i\theta}+\rho e^{-i \theta}), \ 0 \leqslant \theta \leqslant 2\pi \}, \ \rho > 1
\label{bersteindef}
\end{eqnarray}
where $i=\sqrt{-1}$ is the imaginary unit, $\rho$ is the sum of semimajor and semiminor axes. Then we have the following bounds for $\mathcal{L}(\mathcal{E}_{\rho})$, the perimeter of the ellipse, and $\mathcal{D}_{\rho}$, the shortest distance from $\mathcal{E}_{\rho}$ to $[-1,1]$ respectively:
$$\mathcal{L}(\mathcal{E}_{\rho}) \leqslant \pi (\rho+\rho^{-1})^{\frac{1}{2}}, \ \hbox{and} \ \mathcal{D}_{\rho}=\frac{1}{2}(\rho+\rho^{-1})-1.$$
For convenience, we define:
$$M_u=\sup_{z \in \mathcal{E}_{\rho}} |u(z)|.$$
To study the superconvergent property, by introducing the $Hermite's \ contour \ integral$, we have the following point-wise error expression:
\begin{eqnarray}
u(x)-u_N(x)=\frac{1}{2\pi i} \oint_{\mathcal{E}_\rho} \frac{\omega_{N+1}(x)}{z-x} \frac{u(z)}{\omega_{N+1}(z)} dz, \ \forall x \in[-1,1].
\label{def33}
\end{eqnarray}
The following analysis is based on this error expression. \\

\subsection{Theoretical Statements}
Parallel to the conclusion in \cite{Zhao+2016+SISC}, we have the following theorem. \\

\begin{theorem}
Let $0<\alpha<1$. For the interpolation using collocation points as the zeros of Legendre-Lobatto polynomials $\{ x_i \}_{i=0}^N$, the $\alpha$-th Riesz fractional derivative superconverges at $\{ \xi_i^{\alpha} \}$, which satisfies
$$^R D^{\alpha} \omega_{N+1}(\xi_i^{\alpha})=0,\ i=0,1,\ldots,N$$
where $\omega_{N+1}(x)$ is defined by (\ref{def32}).
\end{theorem}
\begin{proof} The proof starts with (\ref{def33}), according to (\ref{def22}), (\ref{lemma24}), (\ref{lemma25}), $\forall x \in[-1,1]$, we have:
\begin{eqnarray}
&&D^{\alpha} (u(x)-u_N(x)) \nonumber \\
&=&\frac{1}{2\pi i} \oint_{\mathcal{E}_\rho} D^{\alpha}( \frac{\omega_{N+1}(x)}{z-x} )\frac{u(z)}{\omega_{N+1}(z)} dz \nonumber \\
&=&\frac{c_1}{2 \pi i} \oint_{\mathcal{E}_\rho} (_{-1}D_x^{\alpha}+{_x D_1^{\alpha}})( \frac{\omega_{N+1}(x)}{z-x} )\frac{u(z)}{\omega_{N+1}(z)} dz \nonumber \\
&=&\frac{c_1}{2 \pi i} \oint_{\mathcal{E}_\rho} \sum_{m=0}^{\infty} \frac{\Gamma(\alpha+1)}{\Gamma(\alpha-m+1)} \frac{(_{-1}D_x^{\alpha-m}+(-1)^k{_x D_1^{\alpha-m}})\omega_{N+1}(x)}{(z-x) ^{m+1}} \frac{u(z)}{\omega_{N+1}(z)} dz \nonumber \\
&=&\frac{1}{2 \pi i} \oint_{\mathcal{E}_\rho}  \sum_{m=0}^{\infty} \frac{\Gamma(\alpha+1)}{\Gamma(\alpha-m+1)} \frac{^R D_o^{\alpha-m} \omega_{N+1}(x)} {(z-x) ^{m+1}}\frac{u(z)}{\omega_{N+1}(z)} dz
\end{eqnarray}
According to the analysis in \cite{Zhao+2016+SISC}, the decay of the error is dominated by the leading term:
$$ \oint_{\mathcal{E}_\rho} \frac{^R D^{\alpha} \omega_{N+1}(x)} {(z-x) }\frac{u(z)}{\omega_{N+1}(z)} dz.$$
When $x=\xi_i^{\alpha}, i=0,1,\ldots,N$, the leading term vanishes, and the remaining terms have higher convergent rates.
\end{proof} \\ \\
Next, we describe a method to compute $\{ \xi_i^{\alpha} \}_{i=0}^N$. For $0<\mu<1$, we start from
$$_{-1}D_x^\mu L_n(x)=\frac{\Gamma(n+1)}{\Gamma(n-\mu+1)}(1+x)^{-\mu} P_n^{\mu,-\mu}(x), \ x \in[-1,1], $$
$$_{x}D_1^\mu L_n(x)=\frac{\Gamma(n+1)}{\Gamma(n-\mu+1)}(1-x)^{-\mu} P_n^{-\mu,\mu}(x), \ x \in[-1,1], $$
to have:
\begin{eqnarray*}
&&^R D^{\alpha} (L_{N-1}-L_{N+1})(x)  \\
&=& c_1(_{-1}D_x^\alpha+_{x}D_1^\alpha)(L_{N-1}(x)-L_{N+1}(x)) \\
&=& c_1(_{-1}D_x^\alpha L_{N-1}(x)+{_{x}D_1^\alpha} L_{N-1}(x) -{_{-1}D_x^\alpha} L_{N+1}(x)-{_{x}D_1^\alpha}L_{N+1}(x)) \\
&=& \frac{\Gamma(N)}{2 \sin(\pi\gamma/2) \Gamma(N-\mu)}[(1+x)^{-\mu} P_{N-1}^{\mu,-\mu}(x)+(1-x)^{-\mu} P_{N-1}^{-\mu,\mu}(x)] \\
&&- \frac{\Gamma(N+2)}{2 \sin(\pi\gamma/2) \Gamma(N-\mu+2)}[(1+x)^{-\mu} P_{N+1}^{\mu,-\mu}(x)+(1-x)^{-\mu} P_{N+1}^{-\mu,\mu}(x)]
\end{eqnarray*}
Then the roots of $^R D^{\alpha} (L_{N-1}-L_{N+1})(x)$ can be computed numerically. \\

\subsection{Numerical Statements and Validations}
\subsubsection{Numerical Validations for Superconvergent Points}
In this subsection, some numerical examples are presented to show the superconvergence. We consider the function $$u(x)=(1+x)^9(1-x)^9.$$
\begin{figure}[htbp] 
   \centering
   \includegraphics[width=4.5in]{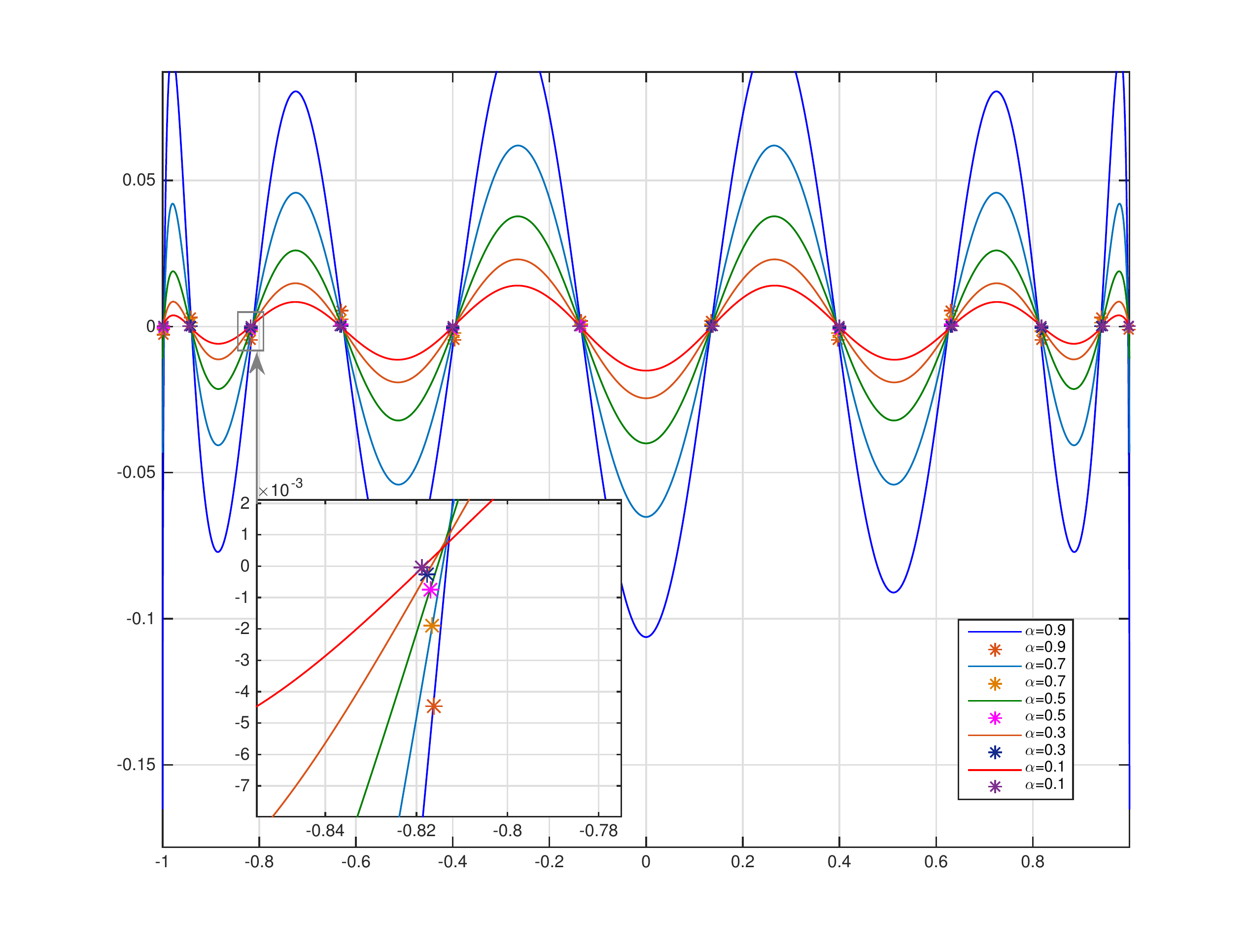}
   \caption{Curves $^RD^\alpha (u-u_{11})(x)$ for different $\alpha$,
   zeros of $^RD^\alpha \omega_{12}(x)$ are highlighted by $*$.}
\end{figure}
Here $u(x)$ is interpolated at $N+1=12$ zero points of $\omega_{12}(x)$, where $\omega_{12}(x)$ is the Legendre-Lobatto polynomial of degree 12. We set $\alpha=0.1$, $0.3$, $0.5$, $0.7$, and $0.9$, respectively. Fig 3.1 plots $^RD^{\alpha}{(u-u_{11})(x)}$ on $[-1,1]$  with different $\alpha$, and the asterisks indicate the superconvergent points, i.e. the zeros of $^RD^{\alpha} \omega_{12}(x)$, predicted by Theorem 3.1. We see the errors at those points are much smaller than the global maximal error. In addition, we observe that both the global maximal error and the errors at the superconvergent points increase, when $\alpha$ increases. \\

\subsubsection{Numerical Observation of Superconvergent Rates}

In order to quantify the superconvergence rate, we define the following ratio:
\begin{eqnarray}
\hbox{ratio}=\frac{\max_{-1\leqslant x \leqslant 1}{\vert ^R D^{\alpha}(u-u_N)(x) \vert}}{\max_{0\leqslant i \leqslant N}{\vert ^R D^{\alpha}(u-u_N)(\xi_i^{\alpha}) \vert}}
\label{ratio1}
\end{eqnarray}
and use $u(x)=(1+x)^9(1-x)^9$ as an example to plot the ratio in the log-log chart with $N = 8, 10, 12, 14, 16$
in Fig. 3.2 (for $\alpha = 0.01, 0.1, 0.2, 0.3, 0.4, 0.5$) and Fig. 3.3 (for $\alpha = 0.5, 0.6, 0.7, 0.8, 0.9, 0.99$).
Two lines $N^2$ and $N^3$ are also plotted as reference slopes.
\begin{figure}[htbp] 
   \centering
    \includegraphics[width=3.5in]{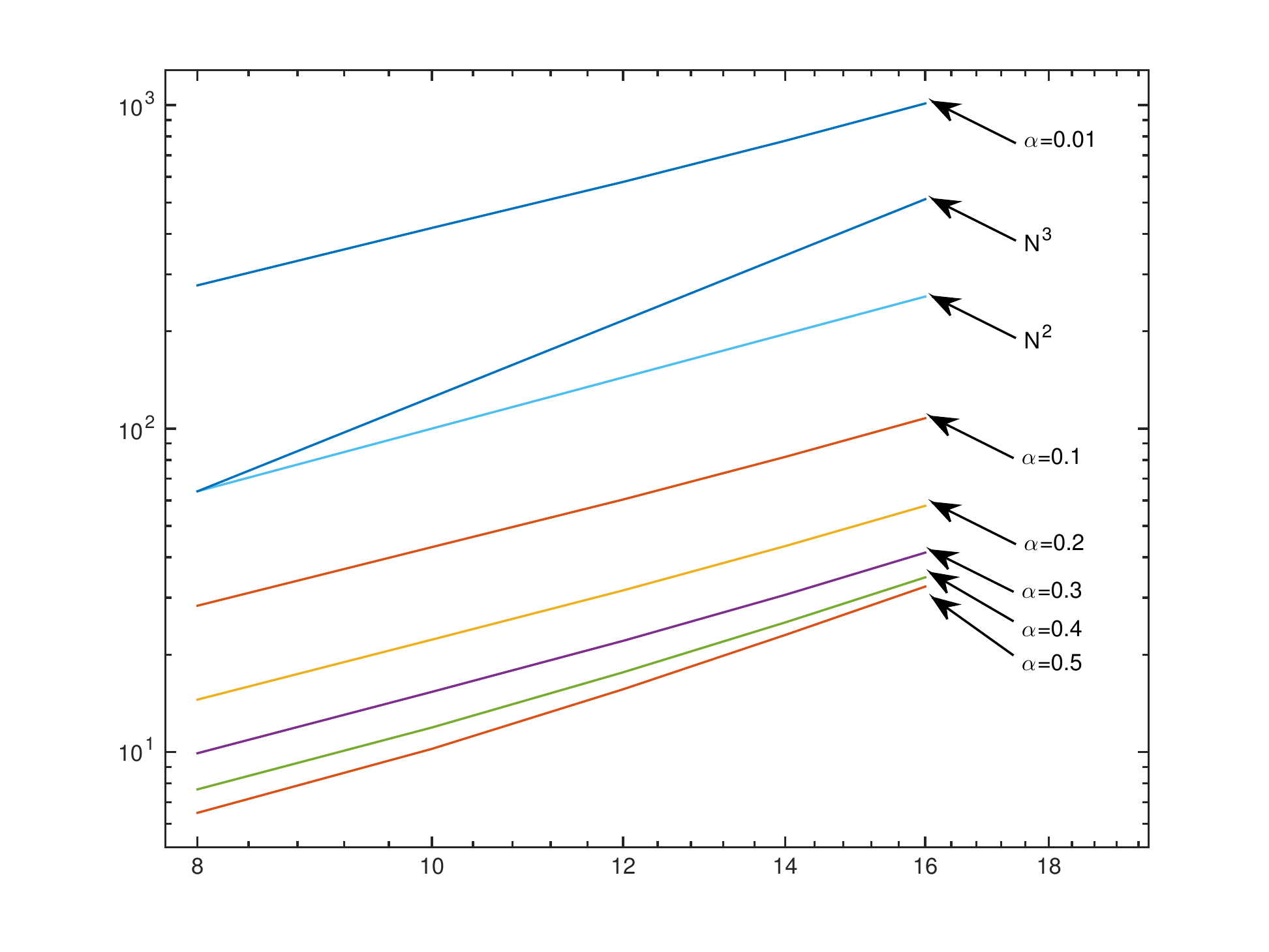}
   \caption{Superconvergent ratios of (3.6) for different $\alpha$-derivatives.}
\end{figure} \\
\begin{figure}[htbp] 
   \centering
    \includegraphics[width=3.5in]{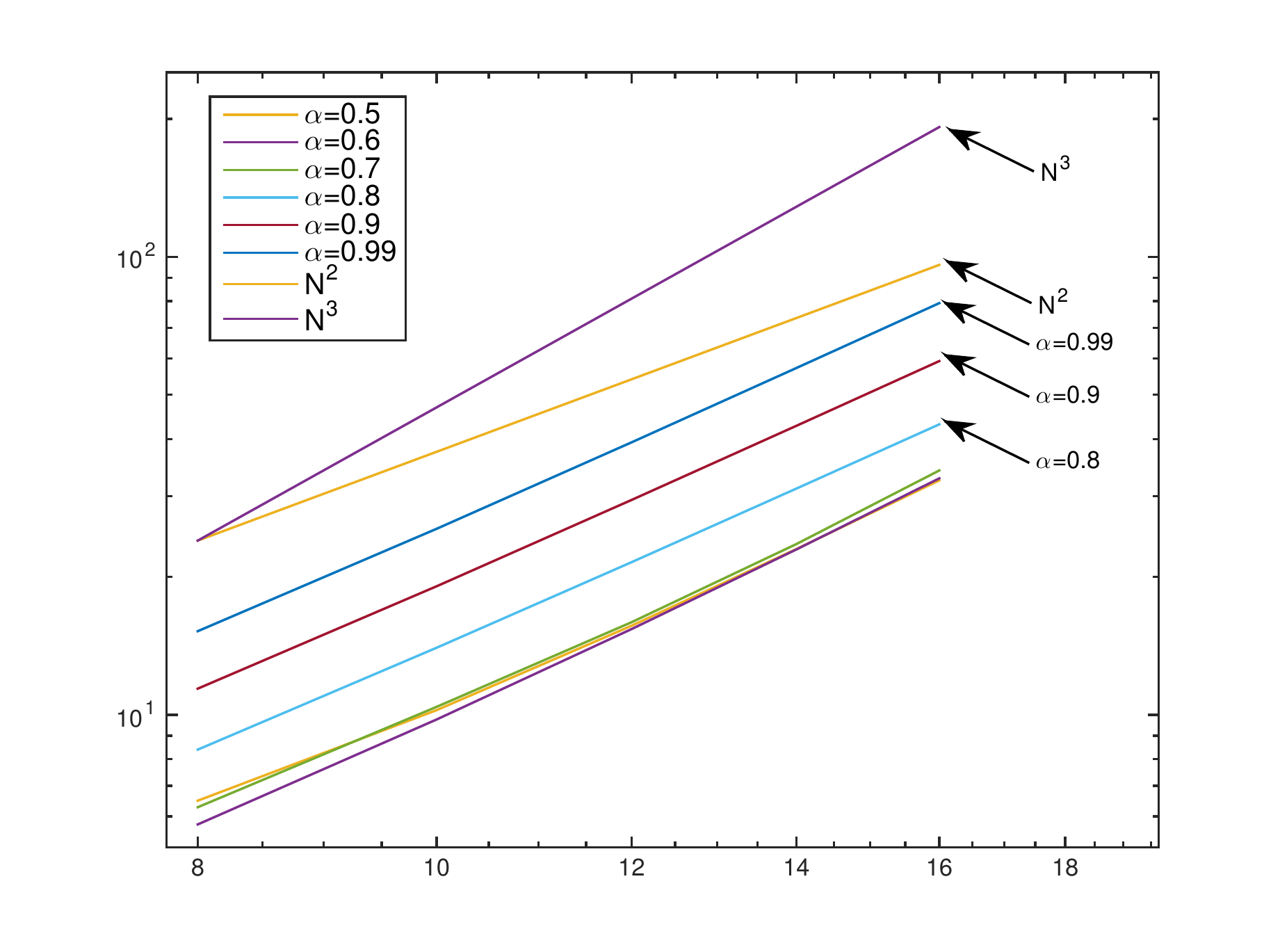}
   \caption{Superconvergent ratios of (3.6) for different $\alpha$-derivatives.}
\end{figure}
We see that at the superconvergent points, the convergent rate is at least $O(N^{-2})$ faster than the global rate. \\

\section{GJF Fractional Interpolation For Arbitrary Positive $\alpha$}
\subsection{Theoretical Statements}
When $\alpha>1$, interpolation by the Lobatto-type polynomials, which provides zeros of multiplicity 1 at $x=\pm 1$ does not work anymore, since it is not able to control the two-sided Riemann-Liouville fractional derivatives of order $\alpha > 1$. Inspired by Lemma 2.8, the GJF fractional interpolation is introduced here. On the other hand, due to singularities at $x=\pm 1$, some more strict conditions are required for $u(x)$.  Let $\alpha>1$ be the order of Riesz fractional derivatives, we always assume that $(1-x^2)^{-\frac{\alpha}{2}}u(x)$ is analytic on $[-1,1]$, and can be analytically extended to a $Berstein \ ellipse$ with an appropriate $\rho$. In this section, we mainly concentrate on the analysis of the situation $0<\alpha<2$. The conclusion can be generalized to the cases $\alpha>2$. Let's start with the definition of GJF fractional interpolation: \\
\begin{definition}
Let $\alpha \in (k-1,k)$ be a given positive real number, where $k \in \mathbb{Z}^+$. Define
\begin{eqnarray}
\alpha^*=\left\{\begin{array}{ll}
\alpha, \ \ 0<\alpha<2, \\
\alpha-k+1, \ \ 2<\alpha \ \hbox{and} \ k \ \hbox{is odd} \\
\alpha-k+2, \ \ 2<\alpha \ \hbox{and} \ k \ \hbox{is even}
\end{array}\right.
\label{def41}
\end{eqnarray}
Suppose $v(x):=(1-x^2)^{-\frac{\alpha^*}{2}}u(x)  \in C[-1,1]$, the goal is to find
\begin{eqnarray}
u_N(x)=(1-x^2)^{\frac{\alpha^*}{2}}v_N(x),
\label{def42}
\end{eqnarray}
where $v_N(x) \in \mathbb{P}_N[-1,1]$, such that
\begin{eqnarray}
u_N(x_k)=u(x_k), \ \  -1\leqslant x_0 <x_1 <\cdots <x_N \leqslant 1,
\label{def43}
\end{eqnarray}
then $u_N(x)$ is called the GJF fractional interpolant with order $\frac{\alpha^*}{2}$ of $u(x)$.
\end{definition} \\
In fact, (\ref{def43}) is equivalent to find $v_N(x) \in \mathbb{P}_N[-1,1]$, such that
\begin{eqnarray}
v_N(x_k)=v(x_k)=(1-x_k^2)^{-\frac{\alpha^*}{2}}u(x_k), \ \  -1\leqslant x_0 <x_1 <\cdots <x_N \leqslant 1.
\end{eqnarray}
Since $v_N(x)$ is a polynomial approximation, by (\ref{def33}), (\ref{def42}), (\ref{def43}), $\forall x \in[-1,1]$, the error can be expressed by:
\begin{eqnarray*}
u(x)-u_N(x)&=&(1-x^2)^{\frac{\alpha^*}{2}}(v(x)-v_N(x)) \\
&=&\frac{1}{2\pi i} \oint_{\mathcal{E}_\rho} \frac{(1-x^2)^{\frac{\alpha^*}{2}}\omega_{N+1}(x)}{z-x} \frac{v(z)}{\omega_{N+1}(z)} dz,
\end{eqnarray*}
then we have:
\begin{eqnarray}
&& ^R D^{\alpha}(u(x)-u_N(x)) \nonumber \\
&=&\frac{1}{2\pi i} \oint_{\mathcal{E}_\rho} {^R D^{\alpha}}[\frac{(1-x^2)^{\frac{\alpha^*}{2}}\omega_{N+1}(x)}{z-x}] \frac{v(z)}{\omega_{N+1}(z)} dz \nonumber \\
&=& \frac{1}{2 \pi i}  \oint_{\mathcal{E}_\rho}   \sum_{m=0}^{\infty} \frac{\Gamma(\alpha+1)}{\Gamma(\alpha-m+1)} \frac{{^R D_{\nu}^{\alpha-m}} [c_{(\frac{\alpha^*+1}{2},N+1)}\mathcal{J}_{N+1}^{-\frac{\alpha^*}{2},-\frac{\alpha^*}{2}}(x)]}{ (z-x) ^{m+1}} \frac{v(z)}{\omega_{N+1}(z)} dz \nonumber \\
\label{proof45}
\end{eqnarray}
where $\nu=o$ when $k$ is odd and $\nu=e$ when $k$ is even. If we set $\{x_i \}_{i=0}^N$ to be zero points of $P_{N+1}^{\frac{\alpha^*}{2},\frac{\alpha^*}{2}}(x)$, then we can consider
$$\omega_{N+1}(x)=c_{(\frac{\alpha^*+1}{2},N+1)}P_{N+1}^{\frac{\alpha^*}{2},\frac{\alpha^*}{2}}(x)=C_{N+1}^{\frac{\alpha^*+1}{2}}(x).$$
Due to the form in (\ref{lemma211}), for $m \in \mathbb{N}$, define:
\begin{eqnarray}
\phi_m(x)= \frac{\Gamma(N+2)}{\Gamma(N+2+\alpha^*)} {^R D_{\nu}^{\alpha-m}} [c_{(\frac{\alpha^*+1}{2},N+1)}\mathcal{J}_{N+1}^{-\frac{\alpha^*}{2},-\frac{\alpha^*}{2}}(x)],
\label{def46}
\end{eqnarray}
$$\phi_m=\Vert \phi_m \Vert_{L^{\infty}[-1,1]}.$$
In the following analysis, we only consider $0<\alpha<2$, so that $\alpha=\alpha^*$.
Then, according to (\ref{gegenprop2}) and Lemma \ref{lemma28}, when $0<\alpha<1$, we have:
\begin{eqnarray}
\phi_0(x) = C_{N+1}^{\frac{\alpha+1}{2}}(x), \ \phi_0=\frac{\Gamma(N+\alpha+2)}{\Gamma(N+2)\Gamma(\alpha+1)} \sim \frac{(N+1)^{\alpha}}{\Gamma(\alpha+1)},
\label{proof47}
\end{eqnarray}
\begin{eqnarray}
\phi_1(x) = \frac{1}{\alpha-1} C_{N+2}^{\frac{\alpha-1}{2}}(x), \ \phi_1 \leqslant \frac{D_{(\alpha-1)/2}}{\alpha-1} (N+1)^{\frac{\alpha-3}{2}},
\label{proof48}
\end{eqnarray}
when $m \geqslant 2$, $\alpha-m <-1$, we have:
\begin{eqnarray}
\phi_m(x) = {_{-1}I_x^{m-1}} [I^{1-\alpha}_o {c \mathcal{J}_{N+1}^{-\frac{\alpha}{2},-\frac{\alpha}{2}}(x)}]= {_{-1}I_x^{m-1}} \phi_1(x),
\label{proof49}
\end{eqnarray}
where $c=\frac{\Gamma(N+2)}{\Gamma(N+2+\alpha)}c_{(\frac{\alpha+1}{2},N+1)}$ , and therefore
\begin{eqnarray}
\phi_m =\max_{-1 \leqslant x \leqslant 1} \{ \frac{1}{\Gamma(m-1)} \int_{-1}^x (x-t)^{m-2}\phi_1(t) dt \}\leqslant \frac{2^{m-1}}{\Gamma(m)} \Vert {\phi_1(x)} \Vert_{L^{\infty}[-1,1]}.
\label{proof410}
\end{eqnarray}
On the other hand, when $1<\alpha<2$, we have:
\begin{eqnarray}
\phi_0(x) =  C_{N+1}^{\frac{\alpha+1}{2}}(x), \ \phi_0=\frac{\Gamma(N+\alpha+2)}{\Gamma(N+2)\Gamma(\alpha+1)} \sim \frac{(N+1)^{\alpha}}{\Gamma(\alpha+1)},
\label{proof411}
\end{eqnarray}
\begin{eqnarray}
\phi_1(x) = \frac{1}{\alpha-1} C_{N+2}^{\frac{\alpha-1}{2}}(x), \ \phi_1=\frac{\Gamma(N+\alpha+1)}{\Gamma(N+3)\Gamma(\alpha+1)} \sim \frac{(N+1)^{\alpha-2}}{\Gamma(\alpha+1)},
\label{proof412}
\end{eqnarray}
and
$$\phi_2(x) = \int_{-1}^x \phi_1(x) dx +C_0 = {_{-1}I_x^{1}} \phi_1(x)+C_0,$$
where
\begin{eqnarray*}
C_0 &=& I^{2-\alpha}_e  [{\frac{\Gamma(N+2)}{\Gamma(N+2+\alpha)}} c_{(\frac{\alpha+1}{2},N+1)} \mathcal{J}_{N+1}^{-\frac{\alpha}{2},-\frac{\alpha}{2}}(-1)] \\
&=& c_2\frac{\Gamma(N+2)}{\Gamma(N+2+\alpha)} \int_{-1}^1 \frac{(1+t)^{\frac{\alpha}{2}}(1-t)^{\frac{\alpha}{2}}C_{N+1}^{\frac{\alpha+1}{2}}(t)}{(1+t)^{\alpha-1}}dt \\
&\leqslant& { \frac{c_2 c_{(\frac{\alpha+1}{2},N+1)}}{N^{\alpha}}} \int_{-1}^1 (1+t)^{1-\frac{\alpha}{2}}(1-t)^{\frac{\alpha}{2}}P_{N+1}^{\frac{\alpha}{2},\frac{\alpha}{2}}(t)dt
= {\frac{c_2 c_{(\frac{\alpha+1}{2},N+1)}}{N^{\alpha}}} \cdot {_{N+1}^{(\frac{\alpha}{2},\frac{\alpha}{2})}c_0^{(\frac{\alpha}{2},1-\frac{\alpha}{2})}}
\end{eqnarray*}
with $c_2=\frac{1}{2 \cos(\pi\gamma/2)\Gamma(\gamma)}$}, and ${_{N+1}^{(\frac{\alpha}{2},\frac{\alpha}{2})}c_0^{(\frac{\alpha}{2},1-\frac{\alpha}{2})}}$ represents the weighted inner product (see \cite{shen+wang+tang+spectral}, pp 76):
\begin{eqnarray*}
{_{N+1}^{(\frac{\alpha}{2},\frac{\alpha}{2})}c_0^{(\frac{\alpha}{2},1-\frac{\alpha}{2})}}&:=&(1,P_{N+1}^{\frac{\alpha}{2},\frac{\alpha}{2}}(x))_{\omega^{\frac{\alpha}{2},1-\frac{\alpha}{2}}} \\
&=& \frac{2\Gamma(N+\frac{\alpha}{2}+2)}{\Gamma(N+\alpha+2)\Gamma(\frac{\alpha}{2}+1)} \sum_{m=0}^{N+1} \frac{(-1)^m \Gamma(N+\alpha+m+2)}{m!(N-m+1)!\Gamma(m+3)},
\end{eqnarray*}
\begin{figure}[htbp] 
   \centering
   \includegraphics[width=3.5in]{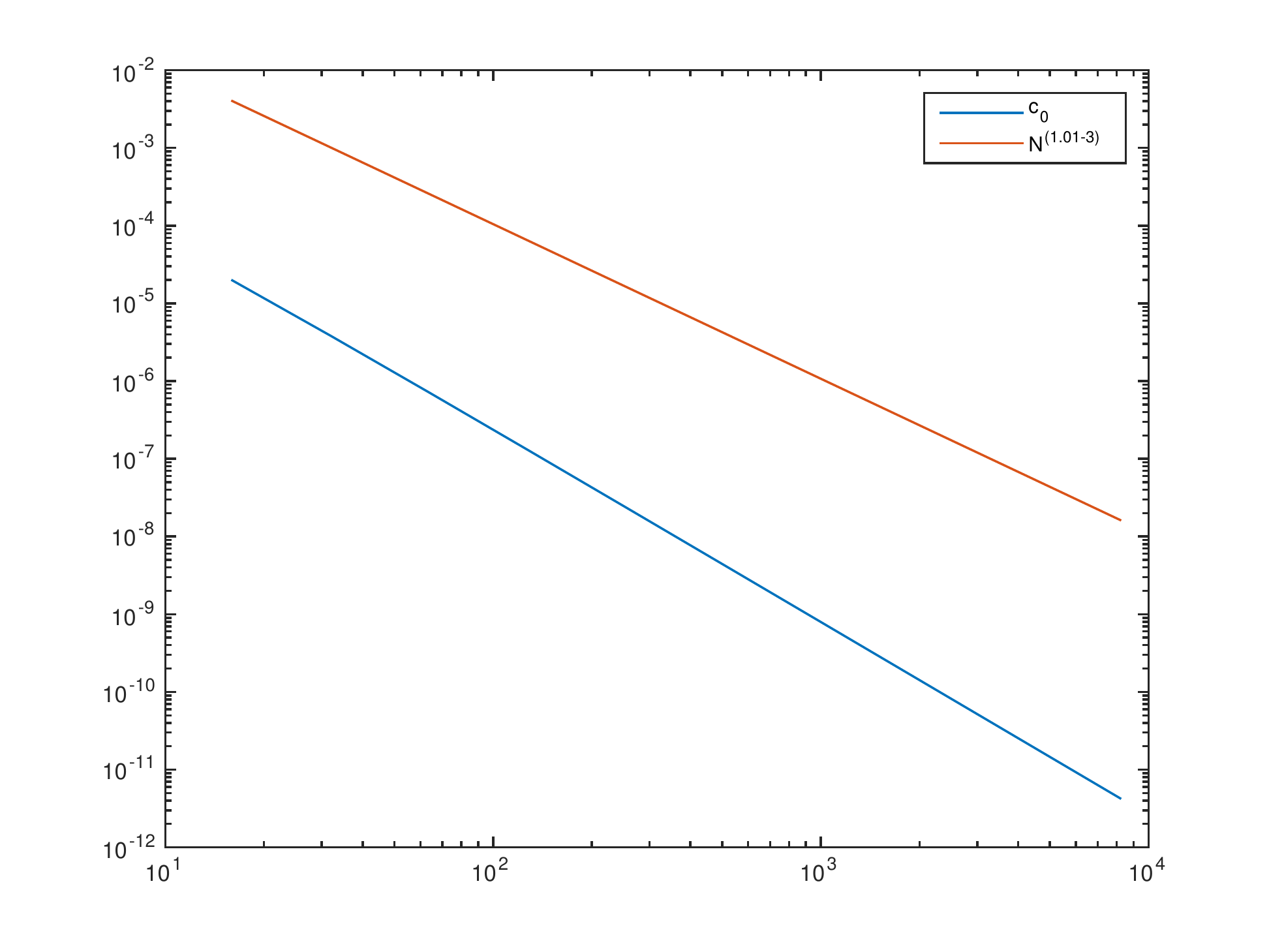}
   \caption{ $\alpha=1.01$, ${_{N+1}^{(\frac{\alpha}{2},\frac{\alpha}{2})}c_0^{(\frac{\alpha}{2},1-\frac{\alpha}{2})}}$ (blue) decays faster than $N^{-1.99}$ (red).}
      \includegraphics[width=3.5in]{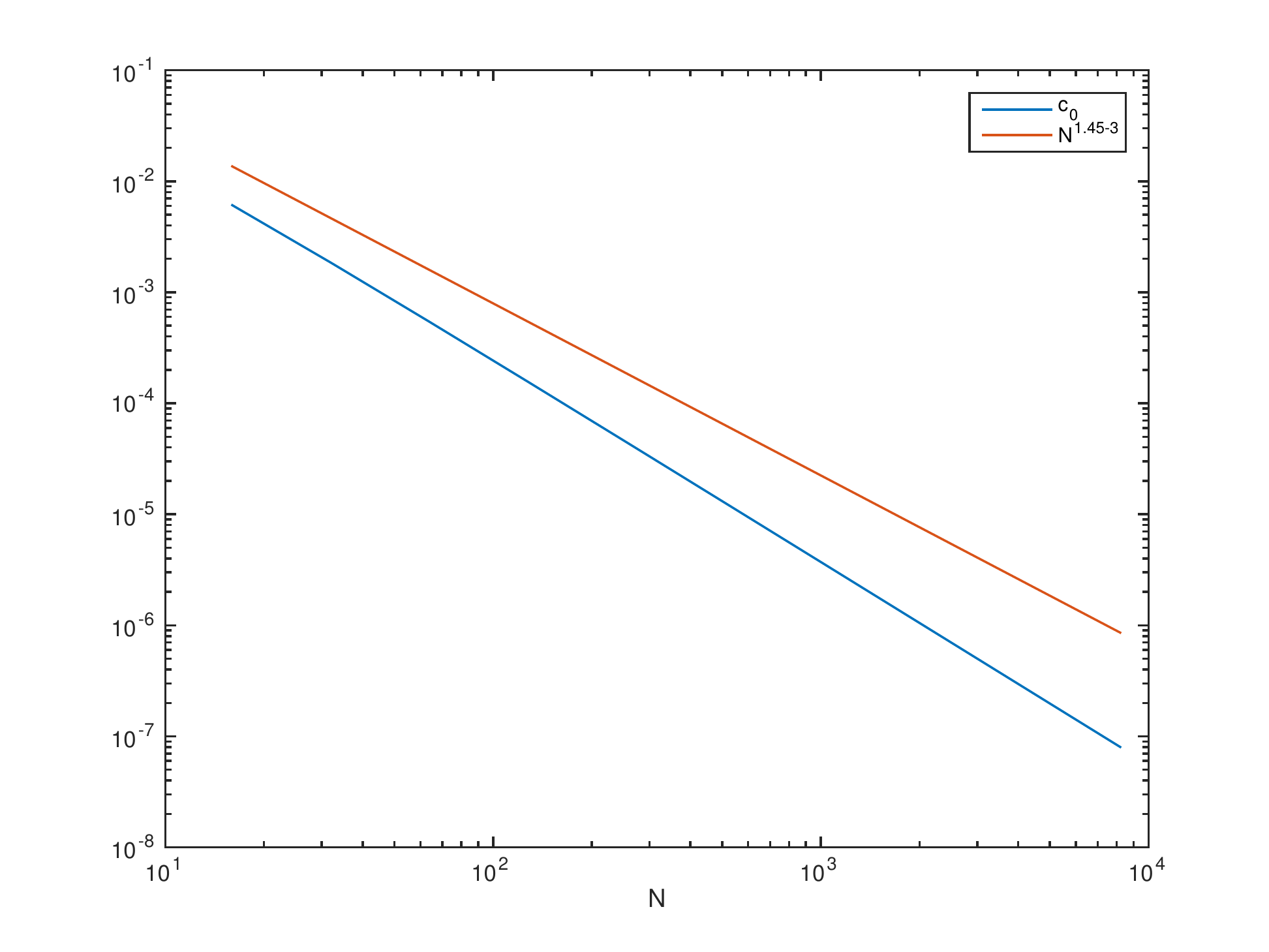}
   \caption{$\alpha=1.45$, ${_{N+1}^{(\frac{\alpha}{2},\frac{\alpha}{2})}c_0^{(\frac{\alpha}{2},1-\frac{\alpha}{2})}}$ (blue) decays faster than $N^{-1.55}$ (red).}
      \includegraphics[width=3.5in]{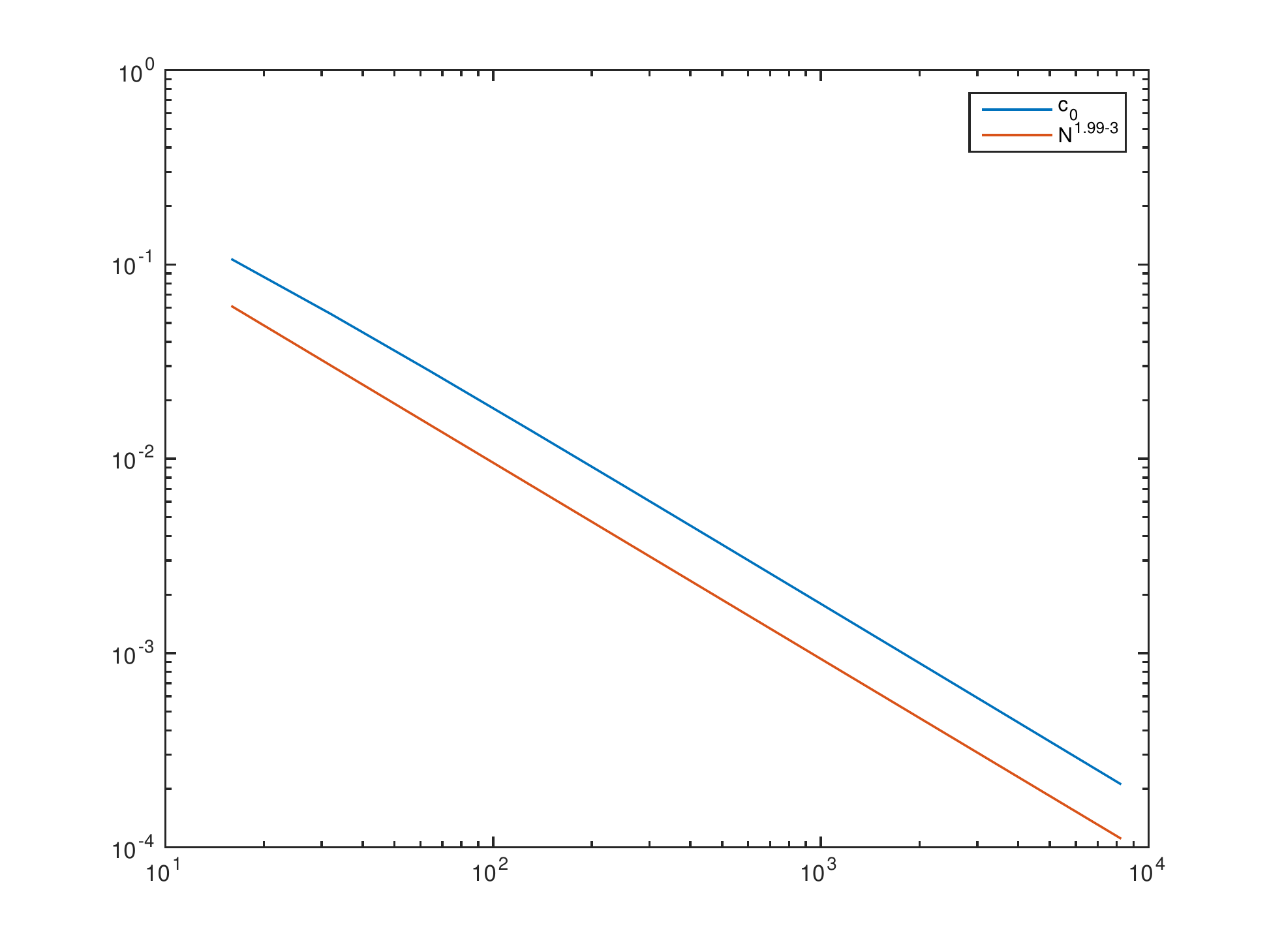}
   \caption{ $\alpha=1.99$, ${_{N+1}^{(\frac{\alpha}{2},\frac{\alpha}{2})}c_0^{(\frac{\alpha}{2},1-\frac{\alpha}{2})}}$ (blue) decays as fast as $N^{-1.01}$ (red).}
\end{figure}
so it is a constant term that can be numerically estimated, seeing Fig. 4.1-4.3:
$${_{N+1}^{(\frac{\alpha}{2},\frac{\alpha}{2})}c_0^{(\frac{\alpha}{2},1-\frac{\alpha}{2})}} \lesssim N^{\alpha-3},$$
by (\ref{property29}), we have:
$$C_0 \lesssim N^{{\frac{1}{2}}\alpha-3}=o(\phi_1);$$
on the other hand, we have:
$$ \vert \int_{-1}^x \phi_1(x)dx \vert \leqslant \frac{1}{\Gamma(1)} \int_{-1}^x \vert \phi_1(t) \vert dt \leqslant (1+x) \phi_1 \leqslant 2 \phi_1$$
therefore,
\begin{eqnarray}
\phi_2 \leqslant 2\phi_1+o(\phi_1),
\label{proof413}
\end{eqnarray}
when $m \geqslant 3$, $\alpha-m <-1$, then we have:
$$ 
{\phi_m(x)}= {_{-1}I_x^{m-2}} [I^{2-\alpha}_e  c\mathcal{J}_{N+1}^{-\frac{\alpha}{2},-\frac{\alpha}{2}}(x)]= {_{-1}I_x^{m-2}} \phi_2(x)$$
so we have
\begin{eqnarray}
\phi_m =\max_{-1 \leqslant x \leqslant 1} \{ \frac{1}{\Gamma(m-2)} \int_{-1}^x (x-t)^{m-3}\phi_2(t) dt \}\leqslant \frac{2^{m-2}}{\Gamma(m-1)} \phi_2.
\label{proof414}
\end{eqnarray}
Then, when $1<\alpha<2$, by (\ref{gegenprop3}), we can continue the estimate of (\ref{proof45}):
\begin{eqnarray}
&& \vert ^R D^{\alpha}(u(x)-u_N(x)) \vert \nonumber\\
&\leqslant& \frac{1}{2 \pi } \frac{\Gamma(N+2+\alpha)}{N+2} \oint_{\mathcal{E}_\rho}   \sum_{m=0}^{\infty}  {\alpha \choose m} \frac{\Gamma(m+1) \vert \phi_m(x) \vert}{\vert z-x \vert ^{m+1}} \frac{\vert v(z) \vert }{\vert C_{N+1}^{\frac{\alpha+1}{2}}(z) \vert} d \vert z \vert \nonumber\\
&\leqslant & \frac{1}{2 \pi } \frac{\Gamma(N+2+\alpha)}{N+2}  \sum_{m=0}^{\infty}  {\alpha \choose m} \frac{\Gamma(m+1)  \phi_m }{(\mathcal{D}_{\rho}) ^{m+1}} \cdot \frac{M_v \cdot \mathcal{L}(\mathcal{E}_{\rho})}{\frac{(N+1)^{\frac{\alpha-1}{2}} \rho^{N+1}}{2 \Gamma(\frac{\alpha+1}{2}) } (1+\rho^{-2})^{-\frac{\alpha+1}{2}}} \nonumber\\
&\leqslant& M_v \frac{\Gamma(N+2+\alpha)\Gamma(\frac{\alpha+1}{2}) (1+\rho^{-2})^{\frac{\alpha+1}{2}} (\rho^2+\rho^{-2})^{\frac{1}{2}}}{\Gamma(N+2) (N+1)^{\frac{\alpha-1}{2}} \rho^{N+1} }
\sum_{m=0}^{\infty}  {\alpha \choose m} \frac{\Gamma(m+1)  \phi_m }{(\mathcal{D}_{\rho}) ^{m+1}}.  \nonumber \\
\label{proof4150}
\end{eqnarray}
We estimate the infinity sum next. When $m \geqslant 2$, by (\ref{proof414}),
$$\frac{\Gamma(m+1)  \phi_m }{(\mathcal{D}_{\rho}) ^{m+1}} \leqslant \frac{\Gamma(m+1) 2^{m-2}}{\Gamma(m-1) (\mathcal{D}_{\rho})^{m+1}} \phi_2=\frac{1}{4 \mathcal{D}_{\rho}} \frac{(m-1)m \cdot 2^{m}}{(\mathcal{D}_{\rho})^{m}} \phi_2$$
if we strictly have $\mathcal{D}_{\rho}>2$, i.e. $\rho>3+2\sqrt{2}$, then $\exists c_{\rho},d_{\rho}>0$, s.t.
$$\frac{2d_{\rho}}{\mathcal{D}_{\rho}}<1, \ \hbox{and} \ (m-1)m \leqslant c_{\rho}(d_{\rho})^m, $$
so we can have:
\begin{eqnarray}
\frac{\Gamma(m+1)  \phi_m }{(\mathcal{D}_{\rho}) ^{m+1}} \leqslant \phi_2 \frac{c_{\rho}}{4 \mathcal{D}_{\rho}} (\frac{2d_{\rho}}{\mathcal{D}_{\rho}})^m .
\label{proof415}
\end{eqnarray}
Therefore, by (\ref{proof415}), the infinity sum in (\ref{proof4150}) can be estimated below:
\begin{eqnarray}
&& \sum_{m=0}^{\infty}  {\alpha \choose m} \frac{\Gamma(m+1)  \phi_m }{(\mathcal{D}_{\rho}) ^{m+1}} = \frac{\phi_0}{\mathcal{D}_{\rho}}+\alpha \frac{\phi_1}{(\mathcal{D}_{\rho})^2}+\sum_{m=2}^{\infty}  {\alpha \choose m} \frac{\Gamma(m+1)  \phi_m }{(\mathcal{D}_{\rho}) ^{m+1}} \nonumber\\
&\leqslant& \frac{\phi_0}{\mathcal{D}_{\rho}}+\alpha \frac{\phi_1}{(\mathcal{D}_{\rho})^2}+ \frac{c_{\rho} \phi_2}{4 \mathcal{D}_{\rho}}\sum_{m=2}^{\infty}  {\alpha \choose m} (\frac{2d_{\rho}}{\mathcal{D}_{\rho}})^m \nonumber\\
&=& \frac{1}{\mathcal{D}_{\rho}}\phi_0+\frac{\alpha}{(\mathcal{D}_{\rho})^2}\phi_1+ \frac{c_{\rho}}{4 \mathcal{D}_{\rho}} (1+\frac{2d_{\rho}}{\mathcal{D}_{\rho}})^{\alpha}  \phi_2
\label{proof416}
\end{eqnarray}
Similarly, when $0<\alpha<1$, by (\ref{proof49}), (\ref{proof410}), (\ref{proof415}), we have:
\begin{eqnarray}
\vert ^R D^{\alpha}(u(x)-u_N(x)) \vert \lesssim  \frac{1}{\mathcal{D}_{\rho}}\phi_0 + \frac{c_{\rho}}{2 \mathcal{D}_{\rho}} (1+\frac{2d_{\rho}}{\mathcal{D}_{\rho}})^{\alpha}  \phi_1.
\label{proof417}
\end{eqnarray}
To sum up the analysis above, we have the following theorem. \\
\begin{theorem}
Let $u(x)$ be a function such that $(1-x^2)^{-\frac{\alpha}{2}}u(x)$ is analytic on and within the complex ellipse $\mathcal{E}_{\rho}$, where $\rho > 3+2\sqrt{2}$, and $u_N(x)$ be the GJF fractional interpolant of $u(x)$ at $\{ \xi_j^{\alpha} \}_{j=0}^N$, which are zero points of $P_{N+1}^{\frac{\alpha}{2},\frac{\alpha}{2}}(x)$. Then, for $0<\alpha<1,$ we obtain the following global error estimation:
\begin{eqnarray}
&&\max_{-1\leqslant x \leqslant 1}{\vert ^R D^{\alpha}(u-u_N)(x) \vert} \nonumber \\
&&\quad \quad \quad  \leqslant  M_v \frac{2\Gamma(\frac{\alpha+1}{2}) (1+\rho^{-2})^{\frac{\alpha+1}{2}} (\rho^2+\rho^{-2})^{\frac{1}{2}}}{\Gamma(\alpha+1)(\rho+\rho^{-1}-2)} (N+1)^{\frac{3\alpha+1}{2}} {\rho}^{-(N+1)}
\label{theorem421}
\end{eqnarray}
and the error estimation at the superconvergent points $\{ \xi_j^{\alpha} \}_{j=0}^N$:
\begin{eqnarray}
&&\max_{0\leqslant j \leqslant N}{\vert ^R D^{\alpha}(u-u_N)(\xi_j^{\alpha}) \vert} \nonumber\\
&&\quad \quad \quad \leqslant c \cdot M_v \frac{2 \Gamma(\frac{\alpha+1}{2}) (1+\rho^{-2})^{\frac{\alpha+1}{2}} (\rho^2+\rho^{-2})^{\frac{1}{2}}}{\Gamma(\alpha+1)(\rho+\rho^{-1}-2)} (N+1)^{\alpha-1} {\rho}^{-(N+1)},
\label{theorem422}
\end{eqnarray}
where $M_v=\sup_{z \in \mathcal{E}_{\rho}} \{ (1-z^2)^{-\frac{\alpha}{2}}u(z) \}$, $c=\frac{c_{\rho} D_{(\alpha-1)/2}}{2(\alpha-1)} (1+\frac{2 d_{\rho}}{\mathcal{D}_{\rho}})^{\alpha}$. \\
For $1<\alpha<2$, the global error estimation and the error estimation at the superconvergent points $\{ \xi_j^{\alpha} \}_{j=0}^N$ are given  respectively in the following:
\begin{eqnarray}
&&\max_{-1\leqslant x \leqslant 1}{\vert ^R D^{\alpha}(u-u_N)(x) \vert} \nonumber \\
&&\quad \quad \quad \leqslant M_v \frac{2\Gamma(\frac{\alpha+1}{2}) (1+\rho^{-2})^{\frac{\alpha+1}{2}} (\rho^2+\rho^{-2})^{\frac{1}{2}}}{\Gamma(\alpha+1)(\rho+\rho^{-1}-2)} (N+1)^{\frac{3\alpha+1}{2}} {\rho}^{-(N+1)};
\label{theorem423}
\end{eqnarray}
\begin{eqnarray}
&&\max_{0\leqslant j \leqslant N}{\vert ^R D^{\alpha}(u-u_N)(\xi_j^{\alpha}) \vert} \nonumber \\
&&\quad \quad \quad \leqslant c \cdot M_v \frac{2 \Gamma(\frac{\alpha+1}{2}) (1+\rho^{-2})^{\frac{\alpha+1}{2}} (\rho^2+\rho^{-2})^{\frac{1}{2}}}{\Gamma(\alpha+1)(\rho+\rho^{-1}-2)} (N+1)^{\frac{3\alpha-3}{2}} {\rho}^{-(N+1)},
\label{theorem424}
\end{eqnarray}
where $c=\frac{\alpha}{\mathcal{D}_{\rho}}+\frac{c_{\rho}}{2} (1+\frac{2 d_{\rho}}{\mathcal{D}_{\rho}})^{\alpha}$.
\end{theorem} \\

\begin{proof}
We can derive (\ref{theorem421}) from (\ref{proof417}), (\ref{proof47}), (\ref{proof48}) and derive
(\ref{theorem423}) from (\ref{proof416}), (\ref{proof411}), (\ref{proof412}), (\ref{proof413}).
As for (\ref{theorem422}) and (\ref{theorem424}),  we have,
$$\phi_0(\xi_j^{\alpha})=P_{N+1}^{\frac{\alpha}{2},\frac{\alpha}{2}}(\xi_j^{\alpha})=0,  \ j=0,\ldots,N,$$
so the first term in (\ref{proof416}) vanishes, i.e.
$$\max_{0\leqslant j \leqslant N}{\vert ^R D^{\alpha}(u-u_N)(\xi_j^{\alpha}) \vert} \leqslant c_{\rho,N} (\frac{\alpha}{(\mathcal{D}_{\rho})^2}\phi_1+ \frac{c_{\rho}}{4 \mathcal{D}_{\rho}} (1+\frac{2d_{\rho}}{\mathcal{D}_{\rho}})^{\alpha}  \phi_2),$$
where $c_{\rho,N}= M_v \frac{\Gamma(N+2+\alpha)\Gamma(\frac{\alpha+1}{2}) (1+\rho^{-2})^{\frac{\alpha+1}{2}} (\rho^2+\rho^{-2})^{\frac{1}{2}}}{\Gamma(N+2) (N+1)^{\frac{\alpha-1}{2}} \rho^{N+1} }$, which establishes (\ref{theorem424}). Similarly, (\ref{theorem422}) follows from (\ref{proof417}), (\ref{proof411}), and (\ref{proof412}).
\end{proof} \\

When $\alpha>2$, since $$^R D^{\alpha}=D^{\alpha-\alpha^*} (^RD^{\alpha^*})$$
where $\alpha-\alpha^*$ is an even integer, we can generalize the results.\\

\begin{corollary}
Let $\alpha \in (k-1,k)$, $k \in \mathbb{Z}^+$, $k \ll N$, and $\alpha^*$ be defined in (\ref{def41}). Under the same assumptions in Theorem 4.2, we have: \\
when k is odd,
\begin{eqnarray}
\max_{-1\leqslant x \leqslant 1}{\vert ^R D^{\alpha}(u-u_N)(x) \vert}  \lesssim (N-k+2)^{\frac{3\alpha^*+1}{2}+2(k-1)} \rho^{-(N-k+2)},
\label{cor431}
\end{eqnarray}
the superconvergent points $\{\xi_{i}^{\alpha} \}_{i=0}^{N-k+1}$ are the zero points of $P_{N-k+2}^{\frac{\alpha^*}{2}+k-1,\frac{\alpha^*}{2}+k-1}(x)$, and at those points, we have:
\begin{eqnarray}
\max_{0\leqslant i \leqslant N}{\vert ^R D^{\alpha}(u-u_N)(\xi_i^{\alpha}) \vert}  \lesssim (N-k+2)^{\frac{3\alpha^*+1}{2}+2(k-2)} \rho^{-(N-k+2)};
\label{cor432}
\end{eqnarray}
when k is even,
\begin{eqnarray}
\max_{-1\leqslant x \leqslant 1}{\vert ^R D^{\alpha}(u-u_N)(x) \vert}  \lesssim (N-k+3)^{\frac{3\alpha^*+1}{2}+2(k-2)} \rho^{-(N-k+3)},
\label{cor433}
\end{eqnarray}
the superconvergent points $\{\xi_{i}^{\alpha} \}_{i=0}^{N-k+2}$ are the zero points of $P_{N-k+3}^{\frac{\alpha^*}{2}+k-2,\frac{\alpha^*}{2}+k-2}(x)$, and at those points, we have:
\begin{eqnarray}
\max_{0\leqslant i \leqslant N}{\vert ^R D^{\alpha}(u-u_N)(\xi_i^{\alpha}) \vert}  \lesssim (N-k+2)^{\frac{3\alpha^*+1}{2}+2(k-3)} \rho^{-(N-k+3)},
\label{cor434}
\end{eqnarray}
where the constants only depend on $\alpha$, $\rho$.
\end{corollary}
\begin{proof}
When $k$ is odd, by (\ref{def22}), (\ref{def46}), and Lemma \ref{lemma28},
$$\phi_m(x)=D^{k-1} \frac{\Gamma(N+2)}{\Gamma(N+2+\alpha^*)} {^R D_{o}^{\alpha^*-m}} [c_{(\alpha^*,N+1)}\mathcal{J}_{N+1}^{\frac{\alpha^*}{2},\frac{\alpha^*}{2}}(x)]$$
From (\ref{gegenprop1}), the leading term is
$$\phi_0(x)=(\alpha+1)\cdots(\alpha-1+2k)C_{N-k+2}^{(\frac{\alpha^*+1}{2}+k-1)}(x), \  \phi_0 \sim N^{\alpha+1+2(k-1)}$$
and the second term is
$$\phi_1(x)=(\alpha+1)\cdots(\alpha-3+2k)C_{N-k+1}^{(\frac{\alpha^*+1}{2}+k-2)}(x), \  \phi_1 \sim N^{\alpha+1+2(k-2)}$$
and $\phi_m=o(\phi_1)$, for $m \geqslant 2$.
Therefore, the leading term vanishes at $\{\xi_{i}^{\alpha} \}_{i=0}^{N-k+1}$, zero points of $P_{N-k+2}^{\frac{\alpha^*}{2}+k-1,\frac{\alpha^*}{2}+k-1}(x)$. Similar to the proof of Theorem 4.2, the estimates (\ref{cor431}) and (\ref{cor432}) are derived from (\ref{proof417}).\\
When $k$ is even,
$$\phi_m(x)=D^{k-2} \frac{\Gamma(N+2)}{\Gamma(N+2+\alpha^*)} {^R D_{e}^{\alpha^*-m}} [c_{(\alpha^*,N+1)}\mathcal{J}_{N+1}^{\frac{\alpha^*}{2},\frac{\alpha^*}{2}}(x)],$$
and the rest of the proof is similar to the case when $k$ is odd.
\end{proof}

\subsection{Numerical Validations}
To make sure $v(x)$ is smooth enough, in this subsection, we consider the function:
$$u(x)=\frac{(1-x^2)^{\frac{\alpha}{2}}}{1+(x+3)^2}$$
\begin{figure}[htbp] 
   \centering
   \includegraphics[width=4in]{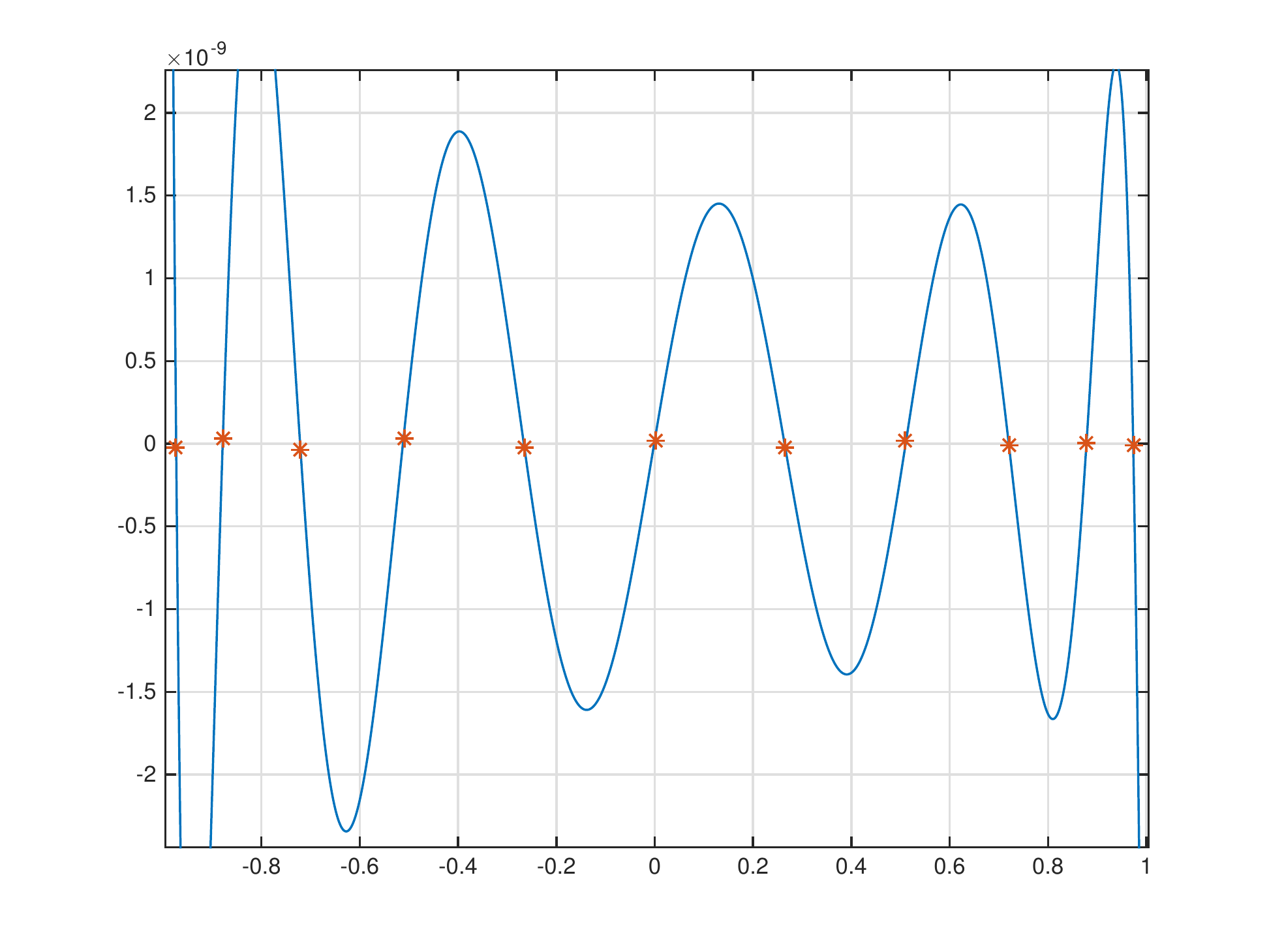}
   \caption{Curves $^RD^\alpha (u-u_{10})(x)$ for the GJF interpolation at 11 pints: $\alpha=0.4$}
\end{figure}
\begin{figure}[htbp] 
   \centering
      \includegraphics[width=4in]{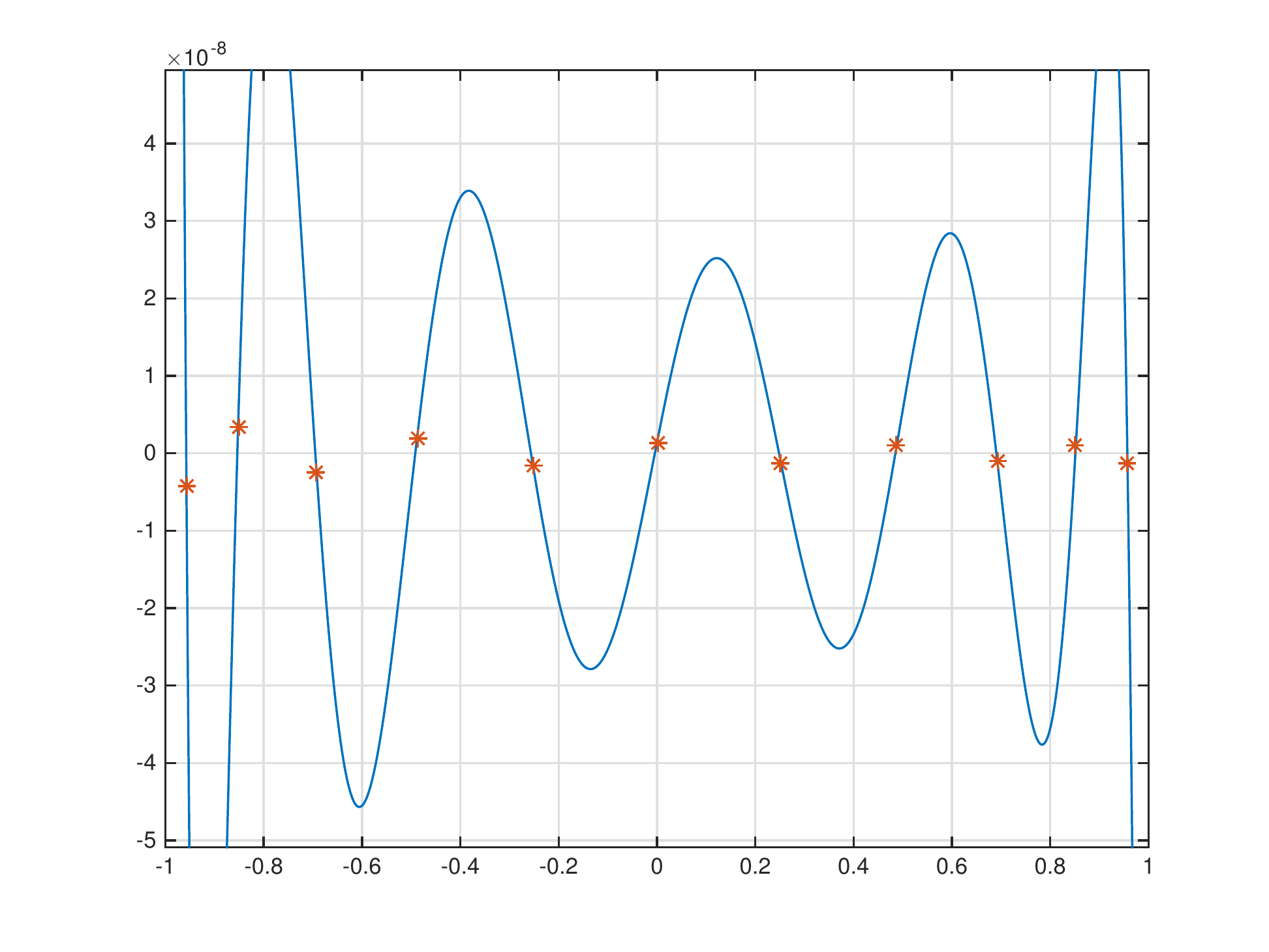}
   \caption{ Curves $^RD^\alpha (u-u_{10})(x)$ for the GJF interpolation at 11 pints: $\alpha=1.7$}
\end{figure}
where $v(x)=\frac{1}{1+(x+3)^2}$, and we set $\alpha =0.4, 1.7$ respectively. It's easy to see that $v(x)$ has two simple poles at $z=-3 \pm i$ in the complex plane. Hence it is analytic within the Berstein ellipse with $\rho > 3+2\sqrt{2}$. In the numerical example, we set $N=10$ so $v_N(x)$ is interpolated at 11 zero points $\{ \xi_j^{\alpha} \}_{j=0}^{10}$ of $P_{11}^{\frac{\alpha}{2},\frac{\alpha}{2}}(x)$. The true solution of $^RD^{\alpha}u(x)$ is approximated by the sum of 40 terms. Fig. 4.4 and 4.5 depict graphs of $^RD^{\alpha}(u-u_{10})(x)$, where $u_N(x)$ is the GJF fractional interpolation, and $\alpha=0.4$  and $1.7$, respectively. According to Theorem 4.2, the 11 interpolation points are predicted as superconvergent points. Similar with Fig. 3.1, the errors at those superconvergent points are significantly less than the global maximal error.

\section{Applications}
In this section, we focus on applications of superconvergence. Let $1<\alpha<2$, and we consider the following FDE:
\begin{eqnarray}
\left\{\begin{array}{ll}
^RD^{\alpha}u(x)+u(x)=f(x),  \ x \in (-1,1)\\
u(-1)=u(1)=0
\end{array}\right.
\end{eqnarray}
We provide two methods to solve for the equation: Petrov-Galerkin method and spectral collocation method. Our goal is to observe superconvergence phenomenon in numerical solutions. In the following numerical examples, we set $f(x)$ be the function such that $u(x)=\frac{(1-x^2)^{\frac{\alpha}{2}}}{1+0.5x^2}$ is the true solution. Then we demonstrate the error curve $^RD^{\alpha}(u-u_N)$ and highlight, by $'*'$, its value at the superconvergent points predicted in Theorem 4.2.

\subsection{Petrov-Galerkin Method}
For any given $1<\alpha<2$, we are looking for
$$u_N \in S_{\alpha}= span \{ \mathcal{J}^{-\frac{\alpha}{2},-\frac{\alpha}{2}}_0, \cdots,  \mathcal{J}^{-\frac{\alpha}{2},-\frac{\alpha}{2}}_N \},$$
such that $\forall v \in \mathbb{P}_N[-1,1]$, we have:
\begin{eqnarray}
(^RD^{\alpha}u_N, v)_{\omega^{\frac{\alpha}{2},\frac{\alpha}{2}}}+(u_N, v)_{\omega^{\frac{\alpha}{2}\frac{\alpha}{2}}}=(f, v)_{\omega^{\frac{\alpha}{2},\frac{\alpha}{2}}}. \label{example4211}
\end{eqnarray}
According to (\ref{lemma212}), by setting $v=P_i^{\frac{\alpha}{2},\frac{\alpha}{2}}$, $i=0,1,\ldots,N$,  (\ref{example4211}) is equivalent to find $(c_0,c_1,\ldots, c_N)^T \in \mathbb{R}^{N+1}$, such that, for $i=0,1,\ldots,N$,
\begin{equation}
\sum_{j=0}^N c_j [d_j (P_j^{\frac{\alpha}{2},\frac{\alpha}{2}}, P_i^{\frac{\alpha}{2},\frac{\alpha}{2}})_{\omega^{\frac{\alpha}{2},\frac{\alpha}{2}}} + (P_j^{\frac{\alpha}{2},\frac{\alpha}{2}}, P_i^{\frac{\alpha}{2},\frac{\alpha}{2}})_{\omega^{\alpha,\alpha}}] =
(f, P_i^{\frac{\alpha}{2},\frac{\alpha}{2}})_{\omega^{\frac{\alpha}{2},\frac{\alpha}{2}}},
\end{equation}
where $d_j=-\frac{\Gamma(j+1+\alpha)}{\Gamma(j+1)}$. We observe that the stiffness matrix is diagonal and dominates the system.

\begin{figure}[htbp] 
   \centering
   \includegraphics[width=4in]{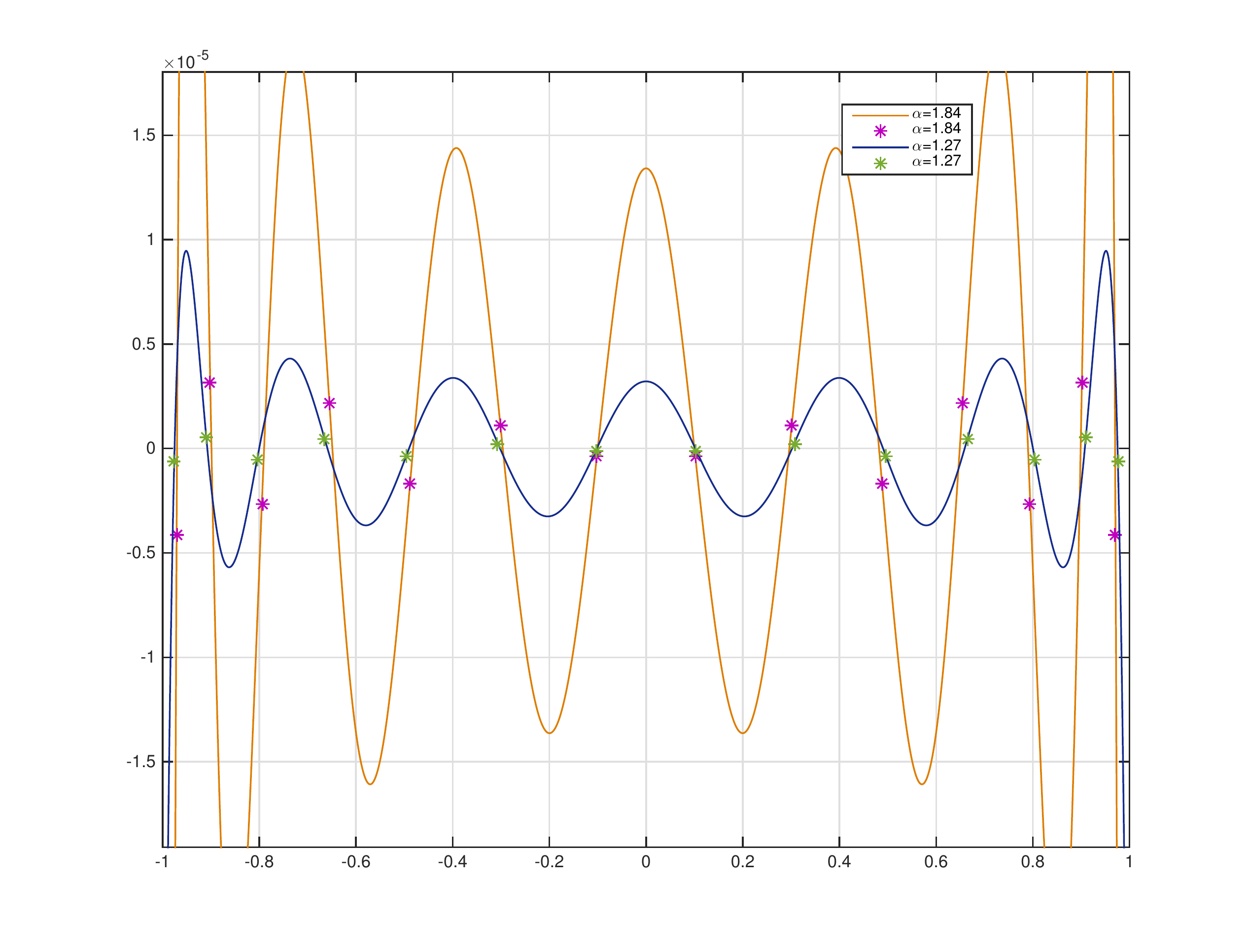}
   \caption{Curves $^RD^\alpha (u-u_{13})(x)$ for the Petrov-Galerkin method: $\alpha=$1.27 and 1.84}
\end{figure}
\begin{figure}[htbp] 
   \centering
      \includegraphics[width=4in]{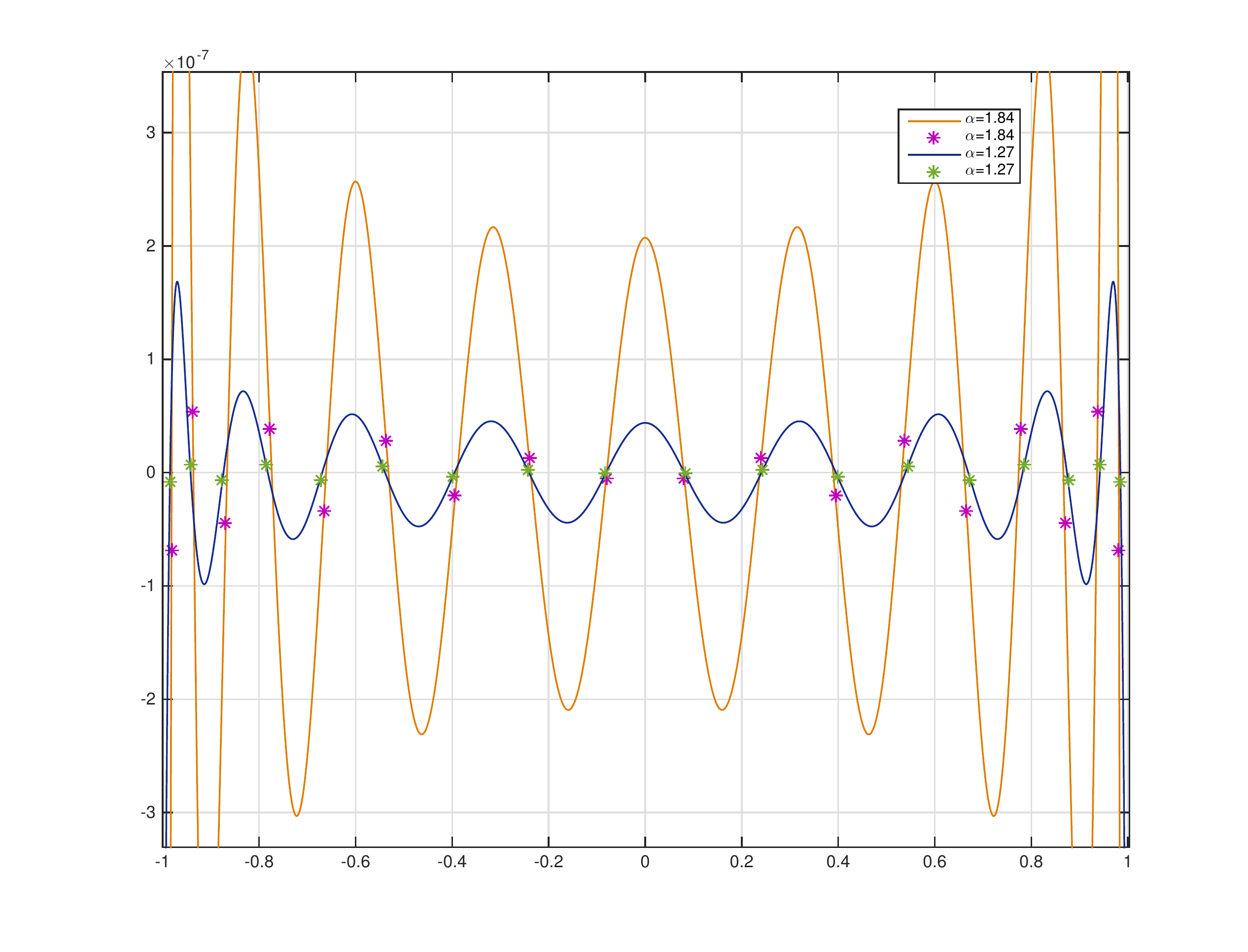}
   \caption{Curves $^RD^\alpha (u-u_{17})(x)$ for the Petrov-Galerkin method: $\alpha=$1.27 and 1.84}
\end{figure}
\begin{figure}[htbp] 
   \centering
      \includegraphics[width=4in]{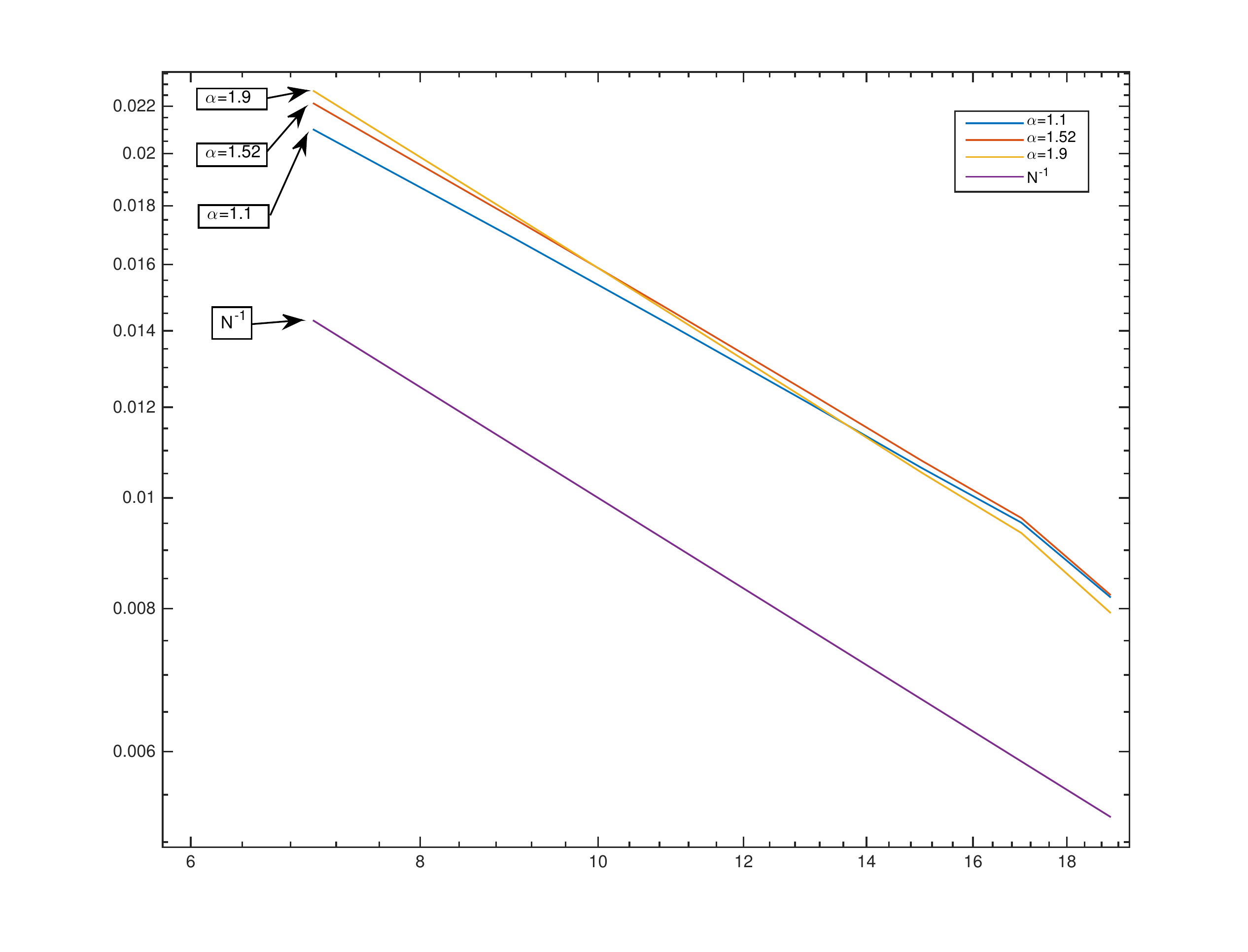}
   \caption{Superconvergent ratios of the Petrov-Galerkin method for different $\alpha$-derivatives.}
\end{figure}

We plotted error curves $^RD^\alpha (u-u_N)$ in Figures 5.1-5.2 for $\alpha=1.27,1.84$, $N=13,17$, respectively. According to Theorem 4.2, the superconvergence points are predicted to be zeros of $P^{\frac{\alpha}{2},\frac{\alpha}{2}}_{N+1}(x)$. We observe that errors at those points are much smaller than the global maximal error, and moreover, both the global maximal error and errors at the superconvergence points increase, when $\alpha$ increases. Fig. 5.3 depicted the reciprocal of (\ref{ratio1}), for $\alpha=1.1,1.52,1.9$, respectively, where $O(N^{-1})$ is plotted as a reference slope (Since they are too close to each other, only three $\alpha$ cases are shown in Fig. 5.3). We see that, at the superconvergence points predicted by Theorem 4.2, the convergence rate is $O(N^{-1})$ faster than the optimal global rate.

\subsection{Spectral Collocation Method}
For any given $1<\alpha<2$, according to the definition of GJF fractional interpolation with order $\frac{\alpha}{2}$, we have:
\begin{eqnarray}
u_N(x)=\sum_{j=0}^N \hat{{\ell}}_j(x) {v_j}:=\sum_{j=0}^N (1-x^2)^{\frac{\alpha}{2}} {\ell}_j(x) {v_j},
\end{eqnarray}
where $\ell_j \in \mathbb{P}_N[-1,1]$ is the Lagrange basis function satisfying
$${\ell}_j(x_i)=\delta_{ij}, \ i,j=0,1,\ldots,N.$$

\begin{figure}[htbp] 
   \centering
   \includegraphics[width=4in]{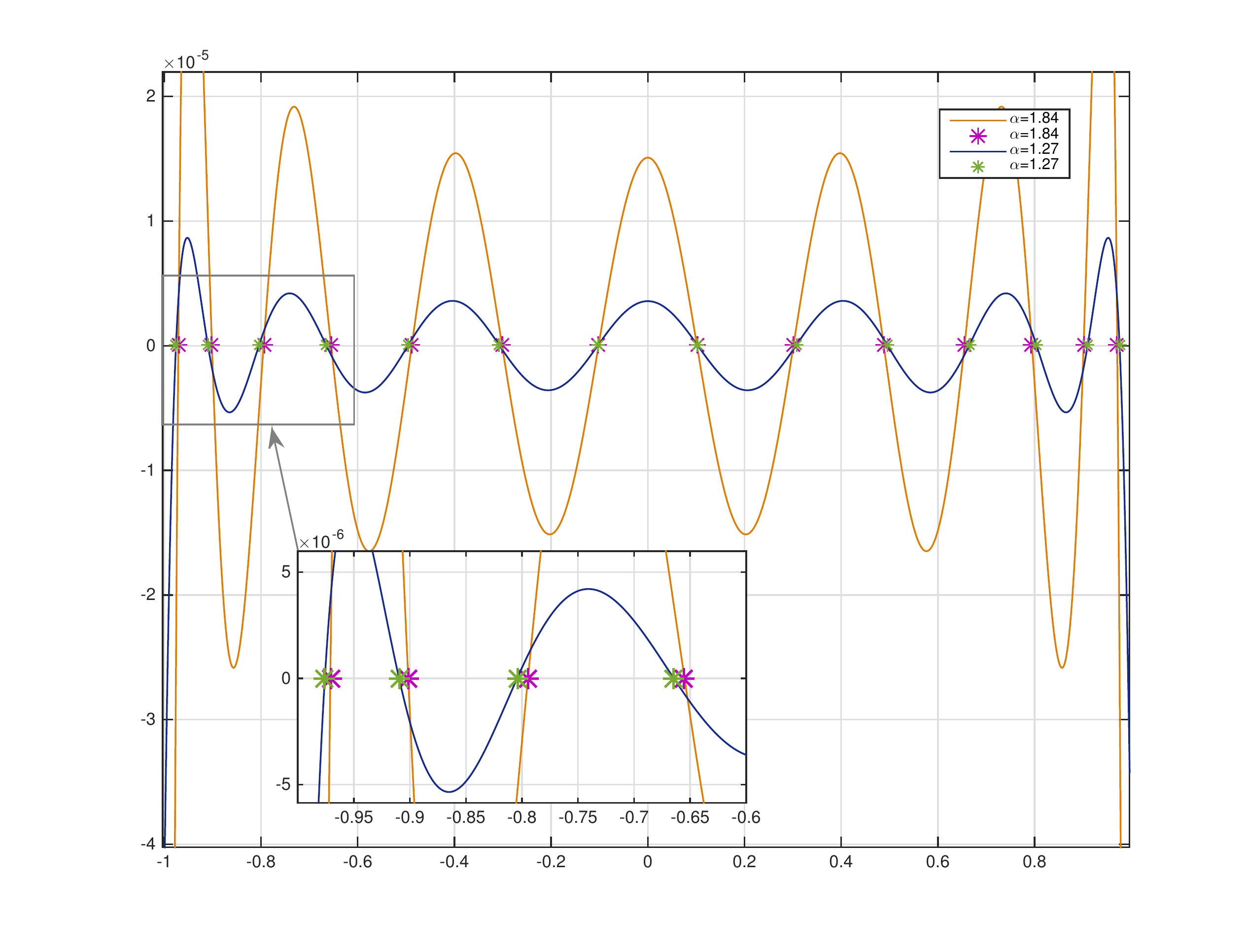}
   \caption{Curves $^RD^\alpha (u-u_{13})(x)$ for the collocation method: $\alpha=$1.27 and 1.84}
\end{figure}
\begin{figure}[htbp] 
   \centering
      \includegraphics[width=4in]{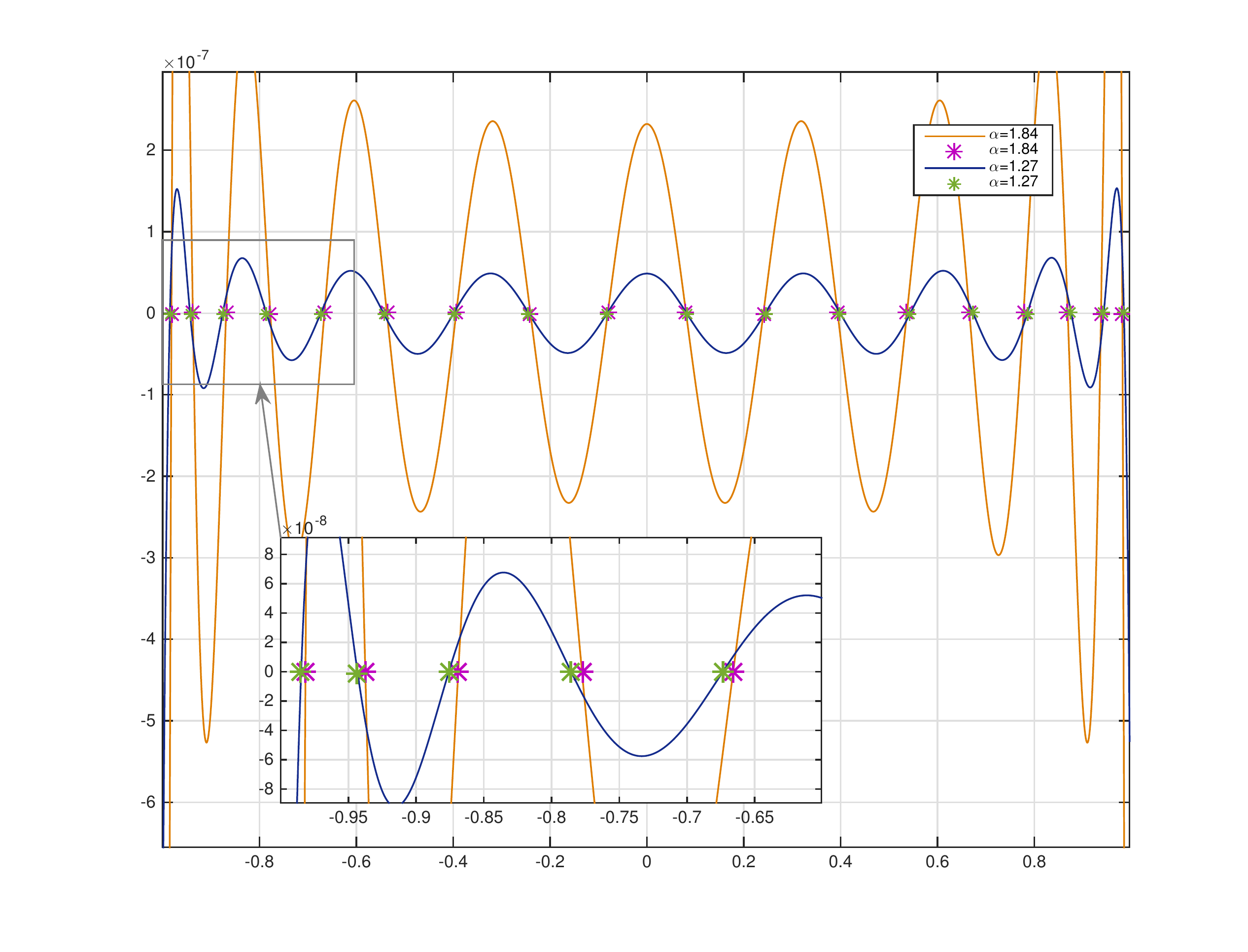}
   \caption{Curves $^RD^\alpha (u-u_{17})(x)$ for the collocation method: $\alpha=$1.27 and 1.84}
\end{figure}
\begin{figure}[htbp] 
   \centering
      \includegraphics[width=4in]{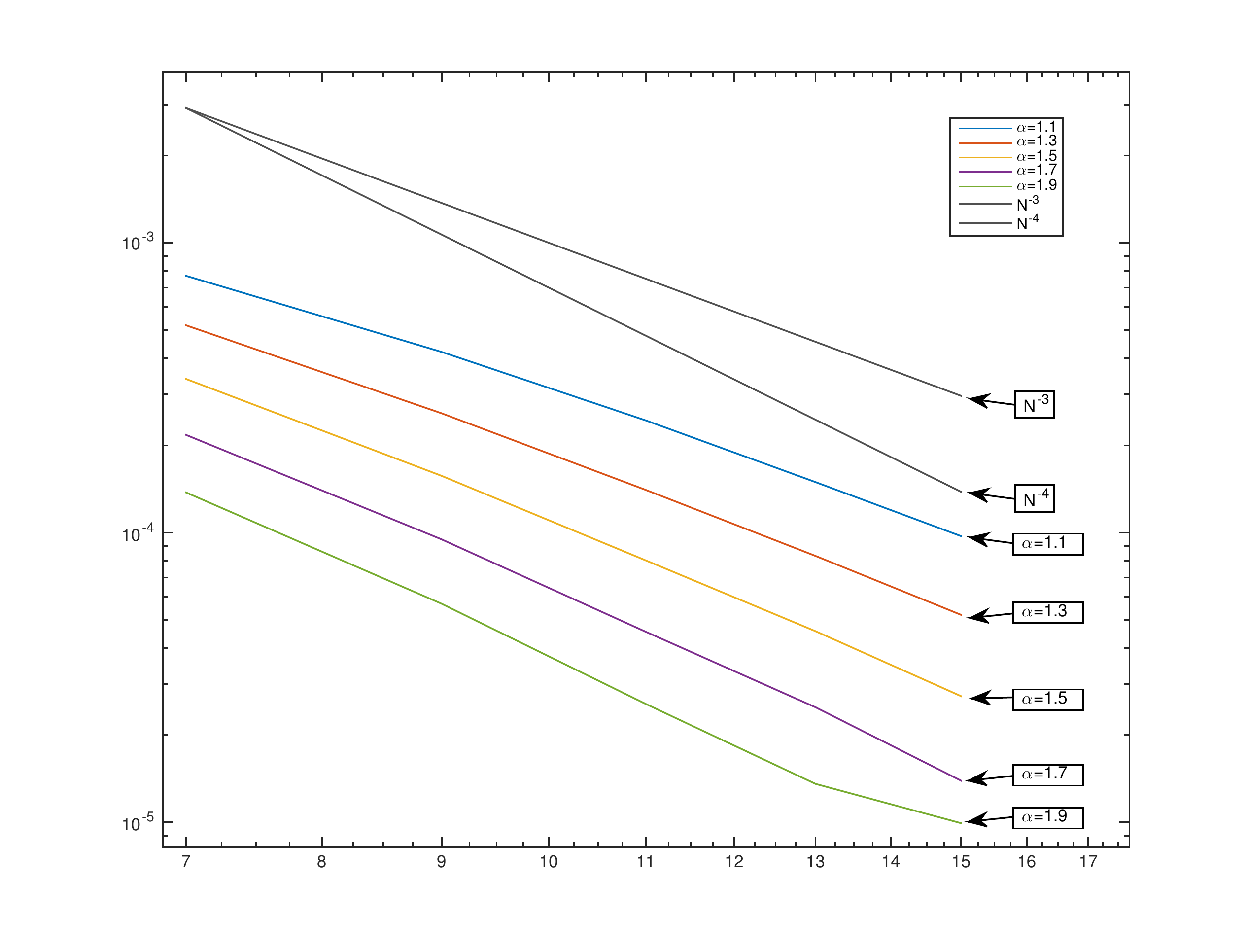}
   \caption{Superconvergent ratios of the collocation method for different $\alpha$-derivatives.}
\end{figure}

Therefore, we are looking for $V=(v_0,v_1, \cdots, v_N)^T \in \mathbb{R}^{N+1}$, such that
\begin{eqnarray}
(^RD^{\alpha}u_N)(x_i)+u_N(x_i)=f(x_i), \quad { i=0,1,\ldots,N,}
\label{example4221}
\end{eqnarray}
where $u_N(x_i)=(1-x_i^2)^{\frac{\alpha}{2}} v_i$. Then, (\ref{example4221}) is equivalent to solve the linear system:
\begin{eqnarray*}
(D+\Lambda)V=F
\end{eqnarray*}
where $D,\Lambda \in \mathcal{M}_{N+1}(\mathbb{R})$, $F \in \mathbb{R}^{N+1}$, and for $i,j=0,1,\ldots,N$,
\begin{eqnarray}
D(i,j)=(^RD^{\alpha}\hat{L}_i)(x_j), \ \ \Lambda(i,i)=(1-x_i^2)^{\frac{\alpha}{2}}, \ \ F(i)=f(x_i),
\end{eqnarray}
and the differential matrix $D$ can be analytically calculated by (\ref{lemma212}). Here, we set $\{ x_i \}_{i=0}^N$ to be zeros of $P^{\frac{\alpha}{2},\frac{\alpha}{2}}_{N+1}(x)$.

Error curves $^RD^\alpha (u-u_N)$ are plotted in Figures 5.4-5.5 for $\alpha=1.27,1.84$, $N=13,17$, respectively. As predicted by Theorem 4.2, the superconvergence points $\{ x_i \}_{i=0}^N$ are zeros of $P^{\frac{\alpha}{2},\frac{\alpha}{2}}_{N+1}(x)$. We can see the errors at those points are significantly smaller than the global maximal error. Furthermore, the performance of superconvergence points of the collocation method is much better than that of the Petrov-Galerkin method. Errors at superconvergence points for the collocation method are more closer to zeros than the Petrov-Galerkin case as demonstrated by Fig. 5.6, the reciprocal of (\ref{ratio1}) ratios with $O(N^{-3})$ and $O(N^{-4})$ as reference slopes. We observe that the convergence rate at superconvergence points for the collocation method is about $O(N^{-3})$ better than the optimal global rate. One possible reason is that the interpolation points and superconvergence points of $\alpha^{th}$ Riesz derivative are identical.

\section{Concluding Remarks}

In this work, we investigated superconvergence for $u-u_N$ under Riesz fractional derivatives. We identified superconvergence points and found the improved convergence rate at those points. When $0<\alpha<1$, we consider $u_N(x)$ as either polynomial interpolation or GJF fractional interpolation, the improvement in convergence rates are $O(N^{-2})$ and $O(N^{-\frac{\alpha+3}{2}})$, respectively. When $\alpha >1$, only the GJF fractional interpolant is discussed due to the singularity, and the improvement in the convergence rate is $O(N^{-2})$. In particular, when $0<\alpha<2$, for the case of GJF fractional interpolation, the superconvergence points are the same as the interpolation points. In addition, when we apply our superconvergence knowledge to  the numerical solution of model FDEs, our theory  predicts accurately the locations of superconvergence points.
Moreover, we notice that for the Petrov-Galerkin method, the convergence improvement at the superconvergence points is only $O(N^{-1})$, which is inferior to $O(N^{-2})$, the improvement for the interpolation; and for the spectral collocation method, the convergence improvement at the superconvergence points is $O(N^{-3})$ to $O(N^{-4})$, which is superior to $O(N^{-2})$, the improvement for the interpolation.

It seems that polynomial-based interpolation plays only a limited role in solving FDEs. We believe that GJF-type fractional interpolation is to be preferred in fractional calculus. We hope that our findings can be useful in numerically solving FDEs, especially when using data at the predicted superconvergence points.


\begin{thebibliography}{10}


\bibitem{Askey+1975}
{\sc R. Askey}, {\em Orthogonal Polynomials and Special Functions}, SIAM, Philadelphia, 1975.
\bibitem{Bernstein+1912}
{\sc S. N. Bernstein}, {\em Sur l'ordre de la meilleure approximation des foncions continues par des polynomes de degr\'e donn\'e}, M\'em. Publ. Class Sci. Acad. Belgique (2), 4 (1912), pp. 1¨C103.

\bibitem{Chen+2014+JCP}{\sc F. Chen , Q. Xu and J.S. Hesthaven}, {A multi-domain spectral method for time-fractional differential equations}, J. Comput. Phys., 293 (2015), pp. 157-172.

\bibitem{Chen+2014+ar}
{\sc S. Chen, J. Shen and L-L. Wang}, {\em Generalized  Jacobi functions and their applications to fractional differential equations},
{  to appear in Math. Comp.} arXiv: 1407. 8303v1

\bibitem{Davis+1975+NY}
{\sc P. J. Davis}, {\em Interpolation and Approximation}, Dover, New York, 1975.

\bibitem{FIAD+Samko}
{\sc S.G. Samko, A.A. Kilbas, O.I. Marichev}, {\em Fractional Integrals and Derivatives, Theory and Applications}, Gordon and Breach Science Publishers, 1993.
\bibitem{Fatone+2015+SISC}
{\sc L. Fatone and D. Funaro }, {\em Optimal Collocation Nodes for Fractional Derivative Operators}, SIAM J. Sci. Comput., 37 (2015), pp. A1504-A1524.

\bibitem{Huang+2014+ar}
{\sc C. Huang, Q. Zhou and Z. Zhang}, {\em Spectral method for substantial fractional differential equations}, arXiv: 1408. 5997v1

\bibitem{Ishteva+2005+MSRJ}
{\sc M. Ishteva, L. Boyadjiev and R. Scherer}, {\em On the Caputo operator of fractional calculus and C-Laguerre functions}, Math. Sci. Res., 9 (2005) pp.161-170.

\bibitem{Li+2012+FCAA}
{\sc C. P. Li, F. H. Zeng, and F.  Liu}, {\em Spectral approximations to the fractional integral and derivative}, Frac.
Calc.  Appl. Anal. 15 (2012) 383-406.


\bibitem{Li+2009+SINUM}
{\sc X. Li and C. J. Xu},{\em  A space-time spectral method for the time fractional diffusion equation}, SIAM J. Numer. Anal. 47 (2009) 2108-2131.

\bibitem{Man+1968+SIAMR}
{\sc B. B. Mandelbrot and J. W. Van Ness}, {\em  Fractional Brownian Motions, Fractional Noises and Applications} SIAM Review, 10  (1968), pp. 422-437.



\bibitem{Mark+2001+PRE}
{\sc   M. M. Meerschaert, D. Benson, B. Baeumer}, {\em Operator L\'{e}vy motion
and multiscaling anomalous diffusion,} Phys. Rev. E 63 (2001)
1112-1117.


\bibitem{Metzler+2000+RP}
{\sc  R. Metzler and J. Klafter},  {\em The random walk's guide to anomalous diffusion: a fractional dynamics approach},  Phys. Reports, 339 (2000) 1-77.

\bibitem{Mus+2012+IMA}
{\sc K. Mustapha and W. McLean}, {\em Uniform convergence for a discontinuous Galerkin, time-stepping method applied to a fractional diffusion equation}, IMA J. Numer. Anal., 32 (2012), pp. 906-925.

\bibitem{Podlubny+1999}
{\sc I. Podlubny}, {\em Fractional differential equations}, Academic Press, New
York, 1999.

\bibitem{Shen+2008+IMA}
{\sc S. Shen, F. Liu, V. Anh and I. Turner}, {\em  The fundamental solution and
numerical solution of the Riesz fractional advection-dispersion
equation}, IMA J. Appl. Math., 73 (2008), pp. 850-872.

\bibitem{sun+2009+pa}
{\sc  H. G. Sun, W. Chen and Y. Q. Chen},
{\em Variable-order fractional differential operators in anomalous diffusion modeling}, Phys. A 388 (2009) 4586-4592.

\bibitem{Stynes+2014+IMA}
{\sc M. Stynes, J. L. Gracia}, {\em A finite difference method for a two-point boundary value problem with a Caputo fractional derivative}, IMA J. Numer.  Anal., (2015) 35, 698-721.

\bibitem{HWang+2013+JCP}
{\sc H. Wang and N. Du}, {\em Fast alternating-direction finite difference methods for three-dimensional space-fractional diffusion equations}, J. Comput. Phys., 258 (2013), pp. 305-318.

\bibitem{LWang+2014+JSC}
{\sc L-L. Wang, X. D. Zhao and Z. Zhang}, {\em Superconvergence of Jacobi-Gauss-type spectral interpolation}, J. Sci. Comput., 59 (2014), pp. 667-687.

\bibitem{Xu+2014+JCP}
{\sc Q. Xu, J. S. Hesthaven}, {\em Stable multi-domain spectral penalty methods for
fractional partial differential equations}, J. Comput. Phys. 257 (2014) 241-258.

\bibitem{Yang+2011+SISC}
{\sc Q. Yang, I. Turner, F. Liu and M. Ili\'c}, {\em  Novel numerical methods for solving the time-space fractional diffusion equation in two dimensions}, SIAM J. Sci. Comput., 33 (2011), pp. 1159-1180.

\bibitem{Zayernouri+2013+JCP}
{\sc M. Zayernouri and G. E. Karniadakis}, {\em Fractional Sturm-Liouville eigen-problems: Theory
and numerical approximations}, J. Comput. Phys., 47 (2013)  2108-2131.

\bibitem{Zayernouri+2014+SISC}
{\sc M. Zayernouri and G. E. Karniadakis}, {\em Fractional spectral collocation method}, SIAM J. Sci. Comput., 36 (2014), pp. A40-A62.


\bibitem{Zeng+2014+SINUM}
{\sc F. Zeng, F. Liu, C. P. Li, K. Burrage, I. Turner, and V. Anh}, {\em Crank-
Nicolson ADI spectral method for the 2-D Riesz space fractional nonlinear reaction-diffusion equation}, SIAM J. Numer. Anal., 52 (2014), pp. 2599-2622.

\bibitem{Zhang+2008+JSC}
{\sc Z. Zhang}, {\em  Superconvergence of a Chebyshev spectral collocation method}, J. Sci. Comput. 34 (2008)  237-246.

\bibitem{Zhang+2012+SINUM}
{\sc Z. Zhang}, {\em  Superconvergence points of polynomial spectral interpolation}, SIAM J. Numer. Anal. 50 (2012), 2966-2985.


\bibitem{Zhao+2016+SISC}
{\sc  X. Zhao, Z. Zhang}, {\em Superconvergence points of fractional spectral interpolation,} SIAM J. Sci. Comput., 38 (2016), pp. A598-A613.

\bibitem{Zheng+2015+SISC}
{\sc M. Zheng, F. Liu, I. Turner and V. Anh}, {\em A novel high order space-time spectral method for the time-fractional Fokker-Planck equation}, SIAM J. Sci. Comput, 37 (2015), pp. A701-A724.

\bibitem{shen+wang+tang+spectral}
{\sc J. Shen, T. Tang, L-L. Wang}, {\em Spectral Methods: Algorithms, Analysis and Applications}, Springer Series in Computational Mathematics, Vol. 41, Springer, 2011.

\bibitem{Lin+2006+Bj}
{\sc Q. Lin and J. Lin}, {\em Finite Element Methods: Accuracy and Improvement}, Math. Monogr. Ser.
1, Science Press, Beijing, 2006.


\bibitem{Mao+2015+ANM}
{\sc Z. Mao, S. Chen and J. Shen}, {\em Efficient and accurate spectral method using generalized Jacobi functions for solving Riesz fractional differential equations}, Applied Numerical Mathematics, 106(2016), pp. 165-181.


\bibitem{Xie+2012+MACOM}
{\sc Z. Xie, L. Wang and X. Zhao}, {\em On exponential convergence of Gegenbauer interpolation and spectral differentiation}, Math. Comp., 82(2012) pp. 1017-1036.




\bibitem{LBW+1995+note}
{\sc L. B. Wahlbin}, {\em Superconvergence in Galerkin Finite Element Methods}, Lecture Notes in Math.
1605, Springer-Verlag, Berlin, 1995.

\bibitem{Roop+2006+AM}
{\sc J. P. Roop}, {\em Computational aspects of FEM approximation of fractional advection dispersion equations on bounded domains in $\mathbb{R}^2$}, J. Comput. Appl. Math, 193(1)(2006), pp. 243-268.

\bibitem{Deng+2012+AM}
{\sc K. Deng, W. Deng}, {\em Finite difference/predictor-corrector approximations for the space and time fractional Fokker-Planck equation}, Appl. Math. Lett., 25(11)(2012) pp. 1815-1821.

\bibitem{Shen+2014+IMA}
{\sc S. Shen, F. Liu, V. Anh, I. Terner, J. Chen}, {\em A novel numerical approximation for the space fractional advection-dispersion equation}, IMA J. Appl. Math. 79(3)(2014), pp. 421-444.

\bibitem{Bu+2014+JCP}
{\sc W. Bu, Y. Tang, J. Yang}, {\em Galerkin finite element method for two-dimensional Riesz space fractional diffusion equations}, J. Comput. Phys. 276(2014), pp. 26-38.

\bibitem{XuanSun+2014+SISC}
{\sc X. Zhao, Z. Sun, Z. Hao}, {\em A fourth-order compact ADI scheme for two-dimensional nonlinear space fractional Schrodinger equation}, SIAM J. Sci. Comput. 36(2014), pp. 2865-2886.

\bibitem{Lei+2013+JCP}
{\sc S. Lei, H. Sun}, {\em A circulant preconditioner for fractional diffusion equations}, J. Comput. Phys. 242(2013), pp. 715-725.

\bibitem{HPang+2012+JCP}
{\sc H. Pang, H. Sun}, {\em Multigrid method for fractional diffusion equations}, J. Comput. Phys. 231(2012), pp. 693-703.


\end{thebibliography}
\end{document}